\documentclass{article}
\usepackage{amssymb,amsmath,amsthm}
\usepackage{latexsym}
\usepackage{pstricks}
\usepackage{multido}

\newcommand{\Q}{\mathbb Q}

\newcommand{\C}{\mathbb C}
\newcommand{\Z}{\mathbb Z}

\newcommand{\F}{\mathbb F}

\newtheorem{lem}{Lemma}[section]
\newtheorem{prop}[lem]{Proposition}
\newtheorem{thm}[lem]{Theorem}

\newtheorem{df}[lem]{Definition}
\newtheorem{cor}[lem]{Corollary}
\newtheorem{conj}[lem]{Conjecture}

\newtheorem{claim}[lem]{Claim}

\newcommand{\LRGen}[1]{
{\langle #1 \rangle}
}

\newcommand{\QRSym}[1]{
{\left(\frac{#1}{p}\right)}
}

\newcommand{\LRVec}[1]{
{\left( #1 \right)}
}

\begin{document}

\title{Generalized Skew Hadamard Difference Sets}
\author{Carlos Salazar-Lazaro}
\maketitle
\begin{abstract}
A skew Hadamard difference set (SHDS) is a difference set that satisfies the skew condition $D+ D^{-1} = G - [1]$. It is known that
if a group $G$ admits a skew hadamard difference set, then $|G| = p^{2\alpha + 1} = 3$ mod $4$.
We will generalize skew Hadamard difference sets to Generalized Skew Hadamard Difference Sets (GSHDS) to cover the case $p=1$ mod $4$, and we will extend all known results that yield necessary existence conditions of skew Hadamard difference sets to our generalization, including the known exponent bounds. 
We will also show a set of necessary existence conditions for the family of groups
$G= (\Z/ p^2\Z)^{2\alpha +1 } \times (\Z/ p\Z)$. These conditions will be given
as integral solutions to non-trivial equations in the group algebra of 
$(\Z/ p\Z)^{2\alpha}$. We will close the article with a general $p$-divisibility
condition of the difference intersection numbers $\nu_{G,L}(h,D) = |D \cap h L| - |D^{(n_0)} \cap h L|$ of a GSHDS for a special class of subgroups $L$.
\end{abstract}

\section{Introduction}


We will assume that the reader is familiar with the theory of
Difference Sets as it can be found in Jungnickel's book in  \cite{DJungnickel}, and in
Lander's book in
\cite{Lander}. Also, we will assume that $G$ is 
an abelian group, and we will use the additive notation for its
group operation. Additionally, for an arbitrary element $A = \sum_{g\in G} a_g [g] \in \Z[G]$ in the group algebra of $G$, we will use 
the notation $A(x)$ to represent,

\begin{eqnarray}
\nonumber
A(x) &=& \sum_{g \in D} a_g x^g
\end{eqnarray}

We will view terms of the form $A(x)$ as elements of the group algebra $\Z[G]$ of $G$. Also, given an automorphism $\sigma \in Aut(G)$,
we define $A(x^{\sigma})$ as,

\begin{eqnarray}
\nonumber
A(x^{\sigma}) &=& \sum_{g \in D} a_g x^{\sigma(g)}
\end{eqnarray}

In particular, 
\begin{eqnarray}
\nonumber
A(x^{-1}) &=& A(x^{\mu_{-1}})\\
\nonumber
&=& \sum_{g \in D} a_g x^{\mu_{-1}(g)}\\
\nonumber
&=& \sum_{g \in D} a_g x^{-g}
\end{eqnarray}
Similarly, by $A(x^n)$ we will mean the sum $\sum_{g \in G} a_g x^{n\cdot g}$, where $n \cdot g = g + \cdots + g$ is $g$ summed $n$ times using the sum operation of $G$.
Using the notation of $A(x)$, we will denote by $[1]$ the term $x^0$ where $0$ is the identity element of $G$.
Additionally, whenever $A(x)$ represents a subset of $G$ (i.e. the $a_g \in \{ 0, 1\}$), we will use the notation
$A^{(n)}$ to denote the set represented by $A(x^n)$.

Skew Hadamard Difference Sets (SHDS) are subsets $D$ of an abelian group $G$ such that in the group algebra $\Z[G]$:
\begin{eqnarray}
\nonumber
D(x)D(x^{-1}) &=& (k-\lambda)[1] + \lambda G(x),\\
\nonumber
 D(x) + D(x^{-1}) &=& G(x) - 1.
\end{eqnarray}

Skew Hadamard difference sets are conjectured to exist only in elementary $p$-abelian groups.

\begin{conj}
Let $G$ admit a SHDS. Then, $G = (\Z/ p\Z)^{2\alpha +1}$ where $p=3$ mod $4$.
\end{conj}

The following is what is known about the existence problem for SHDSs. 
This can be found in Chen, Sehgal, and Xiang's paper in \cite{Xia1}.

\begin{thm}
\label{absprop2}
A SHDS has parameters $(v,k,\lambda) = ( v,
\frac{v-1}{2}, \frac{v - 3}{4} )$. Also, if $D \subset G$ is a SHDS,
and $G$ is abelian, then:
    \begin{enumerate}
    \item The order of $G$ satisfies $v = p^{2\alpha +1}$, where $p=3$ mod $4$,
    \item The set $D$ is invariant under the action of the quadratic residues. That is,
  for $(n,p) = 1$ and $\QRSym{n} = 1$, $D(x^n) = D(x)$,
    \item If $\chi$ is any nonprincipal character, then:
    \begin{eqnarray}
    \nonumber
    \chi(D) &=& \frac{-1 + \epsilon_{\chi} \sqrt{-v} }{2},
    \end{eqnarray}
where $\epsilon_{\chi} \in \{ +1,-1\}$,
    \item If $G = \Z / p^{a_1} \Z \times \Z / p^{a_2}\Z \times \cdots
    \times \Z / p^{a_r} \Z$ where: $a_1 \geq a_2 \geq \cdots \geq a_r$ and $
    a_1 \geq 2$, then $a_1 = a_2$.
    \end{enumerate}
\end{thm}

\begin{thm}
(Johnsen's Exponent Bound) Let $D$ be a SHDS in the abelian $p$-group $G$ of $exp(G) = p^s$ and with $|G|= p^{2\alpha+1}$. Then, $s \leq \alpha +1$.
\end{thm}

\begin{thm}
(Chen-Xiang-Sehgal's Exponent Bound) Let $D$ be a SHDS in the abelian $p$-group $G$ of $exp(G) = p^s$ and with $|G| = p^{2\alpha +1}$. Then, $s \leq \frac{\alpha + 1}{2}$.
\end{thm}

We will introduce the concept of a Generalized Skew Hadamard Difference Set (GSHDS) to generalize the results known for SHDSs.

\begin{df}
\label{GSHDS_def_2_1}
\nonumber A  {\rm Generalized Skew Hadamard Difference Set}
(GSHDS) is an element of the group algebra $D(x) \in \Z[G]$ such
that:
\begin{eqnarray}
\label{GSHDSeqn_2_1} D(x)D(x^{n_0}) &=& (k_0-\lambda)[1] + \lambda G(x),\\
\label{GSHDSeqn_2_2} D(x) + D(x^{n_0}) &=& G(x) - [1],
\end{eqnarray}
where $[1]=x^{0}$ is the multiplicative unit of $\Z[G]$, $k_0 = |
D(x) \cap D(x^{-n_0})|$, and $n_0$ is some Quadratic Non-Residue of
$(\Z/ exp(G) \Z)^*$ for which equation(\ref{GSHDSeqn_2_2}) holds.
\end{df}

The number $n_0$ will be assumed to be a fixed Quadratic Non-Residue of $(\Z/ exp(G) \Z)^*$ for which Equations (\ref{GSHDSeqn_2_1}) and  (\ref{GSHDSeqn_2_2}) hold,
and $\chi_0$ will be assumed to be the principal character of $G$. Equation (\ref{GSHDSeqn_2_2}) will be called the 
``skew condition'' of $D$.

Using the language of GSHDSs, we extend the existence conjecture of SHDS to GSHDS.


\begin{conj}
If $G$ is an abelian group that affords a GSHDS, then $G = (\Z/ p\Z)^{2\alpha + 1}$.
\end{conj}

We will prove the analogue of Theorem \ref{absprop2} for GSHDSs including the known exponent bounds.
We will also show the following existence conditions for the family of groups $G = (\Z/ p\Z) \times (\Z/ p^2\Z)^{2\alpha +1 }$.

\begin{prop}
\label{gshds3_prop1}
Let $G = (\Z/ p\Z) \times H$, $H = (\Z/ p^2\Z)^{2\alpha +1 }$, $L = p \cdot H = (\Z/ p\Z)^{2\alpha + 1}$, and $K = \frac{L}{(\Z/ p\Z)} = (\Z/ p\Z)^{2\alpha}$.
Assume that $D \subset G$ is a GSHDS. Let $D' = H \cap D$, $D'' = L \cap D$. There is an element $L_0(x) \in \Z[K]$ that depends on the structure of $G$, and there
are elements $A(x),B(x) \in \Z[K]$ that depend on $D$ such that,
\begin{enumerate}
 \item The following holds,
\begin{eqnarray}
\nonumber
 L_0(x) * L_0(x^{-1}) &=& p^{4\alpha -1} [1] - p^{2\alpha -1 } K(x)
\end{eqnarray}
 \item If $\alpha = 1$, then $D''$ is a GSHDS in $L$.
 \item The following holds for $A(x),B(x),L_0(x)$,
\begin{eqnarray}
 \nonumber
 p^{2\alpha} A(x) &=& \chi_0(A) K(x) + L_0(x) * B(x^{-1}) \\
 \nonumber
 \chi_0(A) &=& p^{\alpha -1} \epsilon_0 b_0 \\
 \nonumber
 p^{2\alpha} B(x) &=& \chi_0(B) K(x) + p L_0(x) * A(x^{-1})\\
 \nonumber
 \chi_0(B) &=& p^{\alpha} \epsilon_0 a_0
\end{eqnarray} 
Where, $\epsilon_0 \in \{ -1,1\}$, $\chi_0$ is the principal character, $a_0, b_0$ are odd numbers, and all the coefficients of $A(x),B(x)$ are odd.
\end{enumerate}
\end{prop}

The proof of Proposition (\ref{gshds3_prop1}) will yield an explicit formula for $L_0(x)$ in terms of the structure of $G$. We leave as an open problem the explicit
calculation of $L_0$ in terms of $p$ and $\alpha$.

We will then  close this article with a $p$ divisibility condition on the Difference Intersection numbers of $D$.

\begin{df}
Let $G$ be an abelian $p$-group, and $L \subset G$ a subgroup. Assume that $G = g_1 L \cup g_2 L \cup \cdots \cup g_h L$. Then, the Difference
Intersection Number $\nu_{G,L}(g_i,D)$ of $D$ at $g_i$ with respect to $L$ is,
\begin{eqnarray}
 \nonumber
 \nu_{G,L}(g_i,D) &=& |D \cap g_i L| - |D^{(n_0)} \cap g_i L|
\end{eqnarray}
\end{df}

\begin{prop}
Let $D$ be a GSHDS in $G$. Assume that $exp(G) = p^s$, and $k \leq s-1$.
Consider the map $\mu_{p^k} : G \rightarrow G$ by $\mu_{p^k}(g) = p^k \cdot g$. Let $ L = Ker(\mu_{p^k}) = \{ g \in G \mid p^k \cdot g = 0 \}$, and
$H = \mu_{p^k}(G) = p^k \cdot G$. Then,
\begin{enumerate}
 \item There are integer constants $a_{p^k}, b_{p^k}, c_{p^k}$ such that,
\begin{eqnarray}
 \nonumber
  D(x)^{p^k} &=& c_{p^k}[1] + a_{p^k} D(x) + b_{p^k} D(x^{n_0})
\end{eqnarray}
 \item Let $D_H = D \cap H$, assume the coset decomposition $G = L \cup g_1 L \cup \cdots \cup g_r L$. Let $h_i = \mu_{p^k}(g_i)$, and assume that 
$h_{i_1}, \ldots, h_{i_t}$ is a set of $H_1 = (\Z/ exp(H)\Z)^*$ orbit representatives. Then, the following holds in $\Z[H]$,
\begin{eqnarray}
\nonumber
  (a_{p^k} - b_{p^k}) (D_H(x) - D_H(x^{n_0})) &=& \sum_{k=1}^t \nu_{G,L}(g_{i_k},D) ( O_{h_{i_k}}(x) - O_{h_{i_k}}(x^{n_0}) )  \\
 \nonumber
  & & + p^k ( C_H(x) - C_H(x^{n_0}))
\end{eqnarray}
 Where $O_{h_i}$ is the $H_2 = (\Z/ exp(H)\Z)^{*2}$ orbit of $h_i \in H$, and $C_H \in \Z[H]$.
 \item The exact power of $p$ dividing $a_{p^k} - b_{p^k}$ is $p^k$.
 \item The Difference Intersection numbers $\nu_{G,L}(h_i,D)$ are divisible by $p^k$.
\end{enumerate}
\end{prop}

\section{Generalized Skew Hadamard Difference Sets}




We will assume that $G$ is an abelian group, and we 
will use the additive notation for its group operation.
Using the definition of generalized skew Hadamard
difference sets, we will derive the results that will yield a
generalization of theorem \ref{absprop2}. 




The study of generalized skew Hadamard difference sets is motivated by the existence problem of skew Hadamard difference sets and the extension of the known necessary existence conditions
to the case when $p \neq 3$ mod $4$. However, we note that Lam
has considered a similar generalization for cyclic difference sets, where only Equation (\ref{GSHDSeqn_2_1}) is forced to hold, the skew condition is violated, and $G$ is a cyclic group.
We refer the reader to \cite{Lam1} and \cite{Lam2} for Lam's results.

A consequence of the definition of GSHDS is that the algebra generated by $\{ [1], D(x), D(x^{n_0}) \}$ in
the group algebra $\Z[G]$ is a
$3$-dimensional association scheme. We refer the reader to \cite{Bannai1} for an introduction to the field of association schemes. The following result
classifies GSHDSs as SHDSs or partial difference sets with Paley parameters according to the value of $k_0$ of Definition (\ref{GSHDS_def_2_1}). We refer the reader to Jungnickel's book on
difference sets in \cite{DJungnickel} for an introduction to Paley type partial difference sets.

\begin{prop}
\label{prop2_2}
 Let $D$ be a GSHDS. Let $A(D) = \LRGen{ [1], D(x), D(x^{n_0}) }$. Then, 
\begin{enumerate}
 \item The algebra $A(D)$ is a $3$-dimensional association scheme.
 \item The values of $k_0$ are  $0$ or $k$, where $k$ is the size of the set $D$.
 \item If $k_0 = 0$, then $D$ is a Paley type partial difference set with parameters $(v,\frac{v-1}{2}, \frac{v-5}{4}, \frac{v-1}{4})$. 
 \item If $k_0 = k$, then $D$ is skew Hadamard difference set with parameters $(v,\frac{v-1}{2},\frac{v-3}{4})$.
\end{enumerate}
\end{prop}

\begin{proof}
Clearly, by definition of $D$, we have:
\begin{eqnarray}
 \label{gshds_eq1}
 D(x)D(x^{n_0}) &=& k_0 [1] + \lambda D(x) + \lambda D(x^{n_0})\\
 \label{gshds_eq2}
 D(x) D(x) &=& (k-k_0)[1] + (k-1-\lambda) D(x) + (k-\lambda) D(x^{n_0})\\
 \nonumber
 D(x^{n_0}) D(x^{n_0}) &=& (k-k_0)[1] + (k - \lambda) D(x) + (k -1 -\lambda) D(x^{n_0}) 
\end{eqnarray}
Thus, $A(D)$ is a $3$-dimensional algebra. It suffices to show that $\{ D(x^{-1}), D(x^{-n_0}) \} = \{D(x), D(x^{n_0}) \}$ in order to show that $A(D)$ is an
association scheme.
Consider Equation (\ref{gshds_eq1}) as an equation of sets. Clearly, this equation yields:
\begin{eqnarray}
 \nonumber
 k^2 &=& k_0 + \lambda 2 k
\end{eqnarray}
Thus, $k_0 \leq k$ and $k$ divides  $k_0$. This forces $k_0 = 0$ or $k_0 = k$. 

 Assume that $k_0 = 0$.
 This forces $D(x^{-1}) = D(x)$ and $D(x^{-n_0}) = D(x^{n_0})$. Thus, $A(D)$ is an association scheme. Note that Equation (\ref{gshds_eq2}) yields,
  \begin{eqnarray}
  \nonumber
    D(x) D(x^{-1}) &=& k[1] + (k-1 - \lambda) D(x) + (k-\lambda) D(x^{n_0})
  \end{eqnarray}
 Thus, $D$ is a partial difference set with parameters $(v,k, k-1 - \lambda, k-\lambda)$. By viewing Equation (\ref{gshds_eq1}) as sets, we deduce a formula for $\lambda$.
  \begin{eqnarray}
    \nonumber
    k^2 &=& \lambda(v-1)
  \end{eqnarray}
  Thus, $\lambda = \frac{ v- 1}{4}$, and $k-\lambda = \frac{v-1}{4}$ and $k-\lambda -1 = \frac{v-5}{4}$.

 Now, assume that $k_0 = k$.
 This forces $D(x^{-1}) = D(x^{n_0})$ and $D(x^{-n_0}) = D(x)$. Thus, $A(D)$ is an association scheme. Note that Equation (\ref{gshds_eq1}) yields,
  \begin{eqnarray}
   \nonumber
   D(x)D(x^{-1}) &=& k [1] + \lambda D(x) + \lambda D(x^{-1})
  \end{eqnarray}
  Thus, $D$ is a skew Hadamard difference set with parameters $(v,k,\lambda)$. By viewing the previous equation as an equation of sets, we deduce that,
  \begin{eqnarray}
   \nonumber
    k^2 &=& k + \lambda(v-1)
  \end{eqnarray}
   Which yields $\lambda = \frac{v-3}{4}$.

\end{proof}

We note that the well known example of SHDSs given by the quadratic residues in $\F_{p^{2\alpha +1 }}$ whenever $p=3$ mod $4$ also extends to a Paley type partial difference set for the case $p=1$ mod $4$. This is given by Example 1 of Leung and Ma's paper in \cite{LeungMa}. Hence, the set of GSHDSs is a non-empty set.


\begin{cor}
\label{corDual}
The following hold for $A(D)$,
\begin{enumerate}
\item The value $n_0^2$ is a numerical multiplier of $D$. That is, $D^{(n_0^2)} = D$.
\item The character table of $A(D)$ is given by,
\begin{eqnarray}
\nonumber
 [1] &=& e_{\chi_0} + e_{D_1} + e_{D_2}\\
\nonumber
 D(x) &=& k e_{\chi_0} + \chi_1(D) e_{D_1} + \chi_1(D^{(n_0)})  e_{D_2}\\
\nonumber
 D(x^{n_0}) &=& k e_{\chi_0} + \chi_1(D^{(n_0)}) e_{D_1} + \chi_1(D) e_{D_2}
\end{eqnarray} 
Where $\chi_i \in D_i$, $D_i \subset \overline{G}$, $\{e_{\chi_0}, e_{D_1}, e_{D_2} \} $ are the primitive idempotents of $A(D)$, and $\chi_0$ is the principal character.
\item The sets $D_1,D_2$ have size $k$, and $D_2 = D_1^{(n_0)}$.
\item The set $\{\chi_1(D), \chi_1(D^{(n_0)}) \}$ is the set of distinct character values of $D$.
\item The set $D_1$ is a GSHDS. 
\item The set $A(D_1)$ is the dual association scheme of $A(D)$, and $A(D_1)$ is isomorphic to $A(D)$.
\end{enumerate}
\end{cor}

\begin{proof}
Clearly, because $D(x) + D(x^{n_0}) = G(x) - [1]$, we must have $ D(x^{n_0}) + D(x^{n_0^2}) = G(x)-[1]$. Therefore, $D^{(n_0^2)} = D$. Thus, part 1 follows.

Since $A(D)$ is an association scheme, $A(D)$ has three primitive idempotents. Hence, there exist $D_1, D_2 \subset \overline{G}$ such that,
\begin{eqnarray}
\nonumber
 [1] &=& e_{\chi_0} + e_{D_1} + e_{D_2}\\
\nonumber
 D(x) &=& k e_{\chi_0} + a e_{D_1} + be_{D_2}\\
\nonumber
 D(x^{n_0}) &=& k e_{\chi_0} + c e_{D_1} + d e_{D_2}
\end{eqnarray} 

Where,
\begin{eqnarray}
\nonumber
 e_{D_i} &=& \sum \{e_{\chi} \mid \chi \in D_i \},\\
\nonumber
 e_{\chi} &=& \frac{1}{|G|} \sum_{g \in G } \overline{\chi(g)} x^g,
\end{eqnarray}
and $a = \chi_1(D)$, $b = \chi_2(D)$, $c= \chi_1(D^{(n_0)})$, and $d = \chi_2(D^{(n_0)})$ where $\chi_i \in D_i$.
In the standard literature of character theory, $e_{\chi}$ is the primitive idempotent corresponding to $\chi$ in $\C[G]$.

Similarly, we can express the primitive idempotents in terms of the basis $\{ [1], D(x), D(x^{n_0}) \}$.

\begin{eqnarray}
 \nonumber
  v e_{\chi_0} &=& [1] + D(x) + D(x^{n_0}) \\
 \nonumber
  v e_{D_1} &=& |D_1| [1] + \overline{a} D(x) + \overline{c} D(x^{n_0}) \\
  \nonumber
  v e_{D_2} &=& |D_2| [1] + \overline{b} D(x) + \overline{d} D(x^{n_0})
\end{eqnarray}

Clearly, from the previous equations $e_{D_i}^{(n_0)} \in A(D)$. Also, note that $e_{D_1}^{(n_0)} = e_{D_2}$, else the algebra isomorphism of $A(D)$ induced by $n_0$ will be the identity isomorphism which will contradict
the skew condition for $D$ (Equation (\ref{GSHDSeqn_2_2})). Thus, $D_1^{(n_0)} = D_2$ and it follows that $|D_2| = |D_1| = k$. 
Since $D_2 = D_1^{(n_0)}$, it follows that $b= \chi_1 ( D^{(n_0)} )$, $d = \chi_1 ( D)$, and we can choose $\chi_2$ such that $\chi_2 = \mu_{n_0} \cdot \chi_1$ . Where $\mu_{n_0} \in Aut(G)$ is defined by $\mu_{n_0}(g) = n_0 \cdot g$ is the induced automorphism of $G$ by $n_0$ and $\mu_{n_0} \cdot \chi = \chi \circ \mu_{n_0}^{-1}$ is the induced action of $Aut(G)$ on the characters of $G$. This shows parts 2 and 3.

Now we show that $\chi_1(D) \neq \chi_1(D^{(n_0)})$. If $\chi_1(D) = \chi_1(D^{(n_0)})$,  then $v e_{D_1} = k [1] + \overline{\chi_1(D)} ( G(x) - [1])$. Note that, since $\chi_1(D)$ is an
algebraic integer in the ring of integers of $\Q(\eta_{exp(G)})$, it makes sense to talk about the image of $\chi_1(D)$ under the action
of the Galois automorphism $\mu_{n_0}$ of $\Q(\eta_{exp(G)})$ induced by $n_0$. That is, the value $\chi_1(D)^{(n_0)}$. Clearly, by a direct
calculation we can show that $\chi_1(D)^{(n_0)} = \chi_1(D^{(n_0)})$. 
Consider, $\chi_2(D) = \mu_{n_0} \cdot \chi_1 (D) = \mu_{n_0^{-1} * n_0^2} \cdot \chi_1(D) = \mu_{n_0^{-1}} \cdot \chi_1(D^{(n_0^{-2})}) = \mu_{n_0^{-1}} \cdot \chi_1(D^{(n_0^{-2})}) = \mu_{n_0^{-1}} \cdot \chi_1(D)$ as $n_0^2$ is a numerical multiplier of $D$. Hence, $\chi_2(D) = \mu_{n_0^{-1}} \cdot \chi_1(D) =
\chi_1(D^{(n_0)}) = \chi_1(D)$. Now, by applying the Galois automorphism $\mu_{n_0}$ of $\Q(\eta_{exp(G)})$ to the equation $\chi_2(D) = \chi_1(D)$, we get that, $\chi_1(D^{(n_0)}) = \chi_1(D)^{(n_0)} = \chi_2(D)^{(n_0)} = \chi_2(D^{(n_0)})$. Hence, $\chi_2(D^{(n_0)}) = \chi_1(D^{(n_0)}) = \chi_1(D) = \chi_2(D)$,
Thus, $v e_{D_2} = k[1] + \overline{\chi_2(D)} ( G(x) - [1])$ and since $\chi_2(D) = \chi_1(D)$, we have that $v e_{D_2} = v e_{D_1}$. 
 which contradicts the primitivity of the $e_{D_i}$'s. Thus, $\chi_1(D), \chi_1(D^{(n_0)})$ are the distinct
character values of $D$. This shows part 4. 

By the standard theory of association schemes, the set
\begin{eqnarray}
\nonumber
 A(D_1) &=& \LRGen{\chi_0, D_1, D_2 = D_1^{(n_0)} }, 
\end{eqnarray}
is the dual association scheme of $A(D)$. Note that, by construction, both $A(D_1)$ and $A(D)$ have the same
character table. Thus, $A(D_1)$ and $A(D)$ are isomorphic. Thus, $A(D_1)$ and $A(D)$ have the same multiplication tables. In particular, $D_1$ is a GSHDS. This shows parts 5 and 6.

\end{proof}

The set $D_1$ of Corollary (\ref{corDual}) will be denoted by $\overline{D}$, the dual GSHDS of $D$. We will show in Proposition (\ref{PropCharValue}), that
\begin{eqnarray}
\nonumber
\chi(D) &=& \frac{-1 + \epsilon_{\chi} \sqrt{\QRSym{-1}v}}{2} 
\end{eqnarray}
Where $\epsilon_{\chi} \in \{ 1, -1 \}$. Hence, we make the convention that,
\begin{eqnarray}
\nonumber
 D_1 &=& \{ \chi \in \overline{G} \mid \chi(D) = \frac{-1 +\sqrt{\QRSym{-1} v}}{2} \}
\end{eqnarray}
It is important to note by the duality between the association schemes $A(D)$ and $A(\overline{D})$, that the dual GSHDS of $\overline{D}$ is $D$ itself.

\begin{cor}
Let $D$ be a GSHDS in $G$, then the dual of the dual of $D$ is $D$ itself. That is, $\overline{\overline{D}} = D$.
\end{cor}

\begin{cor}
Let $D$ be a GSHDS in $G$, and $\chi, \psi$ be non-principal characters of $G$. Let 
$\Omega(\chi) = \{ \chi(D), \chi(D^{(n_0)}) \}$ and $\Omega(\psi)$ be defined similarly. Then,
$\Omega(\chi) = \Omega(\psi)$ as sets. 
\end{cor}

\begin{proof}
It suffices to show that for arbitrary non-principal character $\chi$, $\Omega(\chi) = \Omega(\chi_1)$ where $\chi_1 \in D_1$. Clearly, if $\chi \in D_1$, we must have
$\Omega(\chi) = \Omega(\chi_1)$. Thus, assume that $\chi \in D_2 = D_1^{(n_0)}$. Then, $\chi = n_0 \cdot \chi_1$. Thus, $\Omega(n_0 \cdot \chi_1) = \{ \chi_1( D^{n_0^{-1}}), \chi_1(D) \} = \Omega(\chi_1)$ because
$D^{(n_0^{-1})} = D^{(n_0)}$.
\end{proof}

\begin{cor}
\label{cor_1}
Let $D$ be a GSHDS in $G$ and $\chi$ a non-principal character of $G$. Then,

\begin{enumerate}
 \item If $k_0 = 0$, then,
  \begin{eqnarray}
    \nonumber
    \chi(D) &=& \frac{-1 \pm \sqrt{v}}{2} 
  \end{eqnarray} 
  \item If $k_0 = k$, then,
  \begin{eqnarray}
    \nonumber
    \chi(D) &=& \frac{-1 \pm \sqrt{-v}}{2} 
  \end{eqnarray} 
\end{enumerate}
\end{cor}

\begin{proof}
 Clearly, Equation (\ref{gshds_eq2}) yields,
 \begin{eqnarray}
 \nonumber
   D(x)D(x)  + D(x) + (k_0 -\lambda) [1] &=& (k - \lambda) G(x)
 \end{eqnarray}
 Thus, $\chi(D)$ is a root of the equation:
 \begin{eqnarray}
   \nonumber
   X^2 + X + (k_0 - \lambda) &=& 0
 \end{eqnarray}
 The result follows from the values of $\lambda$ deduced by the previous proposition.
\end{proof}

\begin{lem}
\label{lem_1}
Let $D$ be a GSHDS in $G$. Then, the order of $G$ cannot be a square.
\end{lem}

\begin{proof}
Let $|G|= v$. Assume otherwise. Consider the case $k_0 = 0$. Then, clearly by the last Corollary (\ref{cor_1}), $\chi(D)$ is a rational algebraic integer. In particular, $\chi(D)$ is invariant
under all Galois automorphisms of $\Q(\eta_{exp(G)})$. Thus, 
\begin{eqnarray}
 \nonumber
 \chi(D)^{(n_0)} &=& \chi(D)
\end{eqnarray}
Since $\chi(D)^{(n_0)} = \chi(D(x^{n_0}))$, we must have by Fourier inversion that $D(x) = D(x^{n_0})$. This is a contradiction of the skew condition of $D$.

Now, consider the case $k_0 = k$. Then, clearly by the last Corollary (\ref{cor_1}), $\chi(D) \in \Q(\sqrt{-1}) = \Q(\eta_4)$. Let $l$ be the order of $\chi$. Clearly, $\chi(D) \in \Q(\eta_l)$. Thus, $\chi(D) \in \Q(\eta_4) \cap \Q(\eta_l) = \Q(\eta_{gcd(l,4)})$. Note that $v = 1$ or $3$ mod $4$. Thus, $gcd(l,4) = 1$. Hence, $\chi(D) \in \Q$. Therefore, for arbitrary non-principal character $\chi$, $\chi(D)$ is a rational algebraic integer. A similar argument as given by the previous case shows that this violates the skew condition of $D$. 

\end{proof}

\begin{cor}
Let $D$ be a GSHDS in $G$. Then, $|G| = p^{2\alpha +1}$ for some odd prime $p$.
\end{cor}

\begin{proof}
It suffices to show that $v$ cannot be composite. The result will follow by the previous Lemma (\ref{lem_1}).
Assume otherwise, that is, that there are two primes $p$ and $s$ that divide $v$. Let $\chi$ be a character
of order $p$ and $\psi$ a character of order $s$. Since $A(D)$ is a association scheme, we must have:
\begin{eqnarray}
 \nonumber
  \Omega &=& \{ \chi(D), \chi(D^{(n_0)} \} \\
  \nonumber
     &=& \{ \psi(D), \psi(D^{(n_0)}) \}
\end{eqnarray}

Clearly, $\Omega \subset \Q(\eta_p) \cap \Q(\eta_s) = \Q(\eta_{gcd(p,s)}) = \Q$.
Thus, $\Omega$ is a set of rational algebraic integers. Clearly, because $A(D)$ is an association scheme,
for an arbitrary non-principal character $\phi$, we must have,
\begin{eqnarray}
 \nonumber
 \{ \phi(D), \phi(D^{(n_0)}) \} &=& \Omega
\end{eqnarray}
Thus, for arbitrary non-principal character $\phi$, $\phi(D)$ is a rational algebraic integer. By using a similar argument as Lemma (\ref{lem_1}), we conclude 
that this contradicts the skew condition of $D$.
\end{proof}

\begin{prop}
\label{PropCharValue}
Let $D$ be a GSHDS in $G$, $|G| =v = p^{2\alpha + 1}$, and $exp(G) = p^s$. Then, 
\begin{enumerate}
\item For an arbitrary non-principal character $\chi$ of $G$, we have,
\begin{eqnarray}
 \nonumber
 \chi(D) &=& \frac{-1 + \epsilon_{\chi} \sqrt{ \QRSym{-1} v } }{2}
\end{eqnarray}
 Where $\QRSym{\cdot}$ is the quadratic residue symbol, and $\epsilon_{\chi} \in \{ -1, 1\}$. 
\item The numerical multiplier group of $D$ is given by the quadratic residues mod $p^s$. 
\end{enumerate}
\end{prop}

\begin{proof}
Part 1 will follow by showing that $k_0 = 0$ if and only if $\QRSym{-1} = 1$ and $k_0 = k$ if and only if $\QRSym{-1} = -1$. The result will follow from Corollary (\ref{cor_1}). Clearly, $k_0 = 0$ if and only if $v = p^{2\alpha + 1} = 1$ mod $4$. Since $v$ is an odd prime power, it must follow that $p = 1 $ mod $4$. But, $p=1$ mod $4$ if and only if $\QRSym{-1} = 1$. Thus, $k_0 = 0$ if and only if $\QRSym{-1} = 1$. A similar argument shows that $k_0 = k$ if and only if $\QRSym{-1} = -1$.

Now we show part 2. Clearly,

\begin{eqnarray}
 \nonumber
 \chi(D) &=& \frac{-1 + p^{\alpha} \epsilon_{\chi}}{2} + p^{\alpha} \epsilon_{\chi} \frac{ - 1 + \sqrt{\QRSym{-1} p}}{2} \\
 \nonumber
 &=& a_{\chi}  + (2a_{\chi} + 1) \omega
\end{eqnarray}
Where $a_{\chi} = \frac{-1 + p^{\alpha} \epsilon_{\chi}}{2}$ and $\omega = \frac{ -1 + \sqrt{\QRSym{-1}p}}{2}$. Clearly, $\chi(D) \in \Q(\omega)$. Also,
note that $\Q(\omega) \subset \Q(\eta_{p^s})$. By Galois theory, there exist $G_2 \subset Gal(\Q(\eta_{p^s}) \mid \Q) = (\Z/ p^s\Z)^*$ such that $\Q(\omega)$ is the fixed field of $G_2$. A direct calculation yields that $G_2$ is the quadratic residues mod $p^s$. Thus, $\chi(D)^{(n)} = \chi(D)$ for $n$ a quadratic residue mod $p^s$. Therefore,
\begin{eqnarray}
 \nonumber
 \chi(D^{(n)}) &=& \chi(D)
\end{eqnarray}
where $n$ a quadratic residue and $\chi$ a non-principal character. By Fourier inversion, we must have $D^{(n)} = D$ for $n$ a quadratic residue.
\end{proof}

\begin{lem}
Let $G$ admit a GSHDS $D$. Then, a nonzero (nonidentity) element $g$ can be the difference of two elements of $D$ in at most 
$\frac{v-1}{4}$ ways.
\end{lem}

\begin{proof}
The lemma follows immediately from parts 3 and 4 of Proposition (\ref{prop2_2}).



\end{proof}

The following result, shown for skew Hadamard difference sets in \cite{PCHBM}, is due to Camion and Mann. The same proof as found in \cite{PCHBM} suffices to deduce the result for GSHDS. We include this proof for completeness. 

\begin{prop}
\label{Camion2}
Let $G$ admit a GSHDS $D$ with $exp(G) = p^s$. Then, $G = \Z/ p^s \Z \times  \Z/ p^s \Z \times \Z/ p^{a_3} \Z 
\times \cdots \times \Z/ p^{a_l} \Z$ where $s \geq a_3 \geq \cdots \geq a_l$.
\end{prop}

\begin{proof}
Assume that $G = \Z/ p^s \Z \times  \Z/ p^{a_2} \Z \times \Z/ p^{a_3} \Z  \times \cdots \times 
\Z/ p^{a_l} \Z$, where $s > a_2 \geq \cdots \geq a_l$.
Let $S = \{ (c_1,c_2,\ldots, c_l) \in G \mid c_1 \text{ is a unit in } \Z/ p^s \Z \}$.
Clearly, $|S| = v - \frac{v}{p}$. Also, by choosing $b = c_1^{-1} $ mod $p$,
$1 + bp^{s-1}$ is a quadratic residue mod $p^s$. Note that, if $g \in S$ then
$g - (1 + b p^{s-1})g = (p^{s-1}, 0,\ldots, 0)$. 

Consider $S\cap D$, clearly this set is closed under the action of the quadratic residues mod
$p^s$. Hence, the element $(p^{s-1}, 0 , \ldots, 0)$ can be written as the difference of two
elements $S\cap D$ at least $|S\cap D|$ times.

We proceed to calculate $|S \cap D|$ by observing that $S$ is closed under the action of $G_1 = (\Z/p^s \Z)^*$.
Let $G_2 = (\Z/p^s \Z)^{*2}$, then every $G_1$ 
orbit of $S$ splits into two equally sized orbits of $G_2$. Since $D$ is the
sum of $G_2$ orbits where each $G_2$ orbit is picked from exactly one $G_1$ orbit;\footnote{
This deduction follows from the fact that $D$ is invariant under the action of $G_2$,
the quadratic residues, and the fact that $D(x) + D(x^{n_o}) = G(x) - [1]$.
}
we
must have $|S \cap D| = \frac{|S|}{2}=
\frac{v - \frac{v}{p}}{2}$. 

Thus, the element $(p^{s-1}, 0, \ldots,0)$ can be written as the difference of two elements of
$S\cap D \subset D$ at least $\frac{ v - \frac{v}{p}}{2}$ times. By the previous lemma, this
number cannot exceed $\frac{v-1}{4}$. Therefore, we must have:
\begin{eqnarray}
\nonumber
 \frac{v - \frac{v}{p}}{2} \leq \frac{v-1}{4}.
\end{eqnarray} 
Hence, we deduce $v (1 - \frac{2}{p}) \leq -1$. Thus, $ 1- \frac{2}{p} < 0$,
implying $p < 2$. Clearly a contradiction.
\end{proof}







\section{Methodology For Showing The Exponent Bounds}


In our study of GSHDSs, we introduce the concept of Quadratic Residue Slices (QRS) as they were introduced by Chen-Sehgal-Xiang in \cite{Xia1} for SHDSs.

\begin{df}
Let $G$ be an abelian $p$-group of order $v$, and assume $exp(G) = p^s$. A Quadratic Residue Slice (QRS) is a set $D \subset G$ that is invariant under the action
of the quadratic residues mod $p^s$ and has size $\frac{v-1}{2}$. That is, a set $D$ such that,
\begin{eqnarray}
 \nonumber
 D(x) + D(x^{n_0}) &=& G(x) - [1]\\
\nonumber
 D(x^{n}) &=& D(x) \text{ for all } n \in (\Z/ p^s\Z)^{*2}
\end{eqnarray}
Where $n_0$ is a fixed non-quadratic residue mod $p^s$.
\end{df}

Using a similar calculation as Chen-Sehgal-Xiang in \cite{Xia1}, we calculate the character value of a general QRS.

\begin{prop}
Let $D$ be a QRS in an Abelian $p$-group $G$, and $\chi$ a nonprincipal character. Then,
\begin{eqnarray}
 \nonumber
  \chi(D) &=& \frac{-1 + (2a_{\chi} + 1) \sqrt{\QRSym{-1}p}}{2}
\end{eqnarray}
Where, $a_{\chi} \in \Z$. We will call the number $2a_{\chi} + 1$ the Difference Coefficient of $D$ at $\chi$, and we will denote this number as $d_G(D,\chi)$.
\end{prop}

The calculation of the character values of a QRS motivates the definition of the $A_{G,G_1}$ incidence structure.

\begin{df}
\label{AGG1Def}
 Let $G$ be an abelian 
$p$-group with exponent $p^{s}$.
 Let $\theta : G\rightarrow \overline{G}$ be an noncanonical isomorphism such that $\theta(g)(g') =
\theta(g')(g)$. Let $X = \{ \Omega_{g_1}, \cdots , \Omega_{g_r} \} $ be the orbits
of $G_1 = (\Z/ p^s \Z)^*$ on $G\backslash \{ 0\}$ and $Y = \{ \Omega_{\chi_1}, \cdots, \Omega_{\chi_r} \}$ be the 
orbits of $G_1 = (\Z/ exp(G) \Z)^*$ on $\overline{G} \backslash \{ \chi_0 \}$, where $\chi_i = \theta(g_i)$. 
Then, $\rm A_{G,G_1}$ is defined
on the set $Y \times X$ using $\theta$ by:
\begin{eqnarray}
\nonumber
A_{G,G_1} (\Omega_{\theta(g')},\Omega_{g}) &=& \left\{ 
	\begin{array}{lcr}
		(\frac{n}{p})o(p\cdot g) &  \text{  if } \theta(g')(g) = \eta_p^n,\\
		0 &  \text{  else},
	\end{array}
	\right.
\end{eqnarray}
where $(\frac{n}{p})$ is the quadratic residue symbol, $o(g)$ is the order
of $g$ in the group $G$, $\chi_0$ is the principal character, $\eta_p$ is 
a fixed primitive $p$th root of unity, and $p \, \cdot \, g  = p \cdot g = g\, +\, g\, +\, \cdots\, +\, g$ is $g$ added to itself
$p$ times.
\end{df}

With every QRS $D$, we can associate a vector $d$ of $\pm 1$s in $\Z^{|X|}$. The significance of $A_{G,G_1}$ is given by the fact that $A_{G,G_1} d$ yields all the Difference Coefficients of $D$ in $G$.
We will show the following important properties of the incidence structure $A_{G,G_1}$.
\begin{prop}
\label{prop3}
Let $G$ be an abelian $p$-group of $exp(G) = p^s$ and order $p^{\beta}$. Let $D$ be a QRS and $d$ its corresponding vector of $\pm 1$s. Then,
\begin{enumerate}
 \item The matrix $A_{G,G_1}$ satisfies $A_{G,G_1}^2 = \frac{|G|}{p} I$.
 \item If $p^k$ divides all the entries of $A_{G,G_1} d$ then $p^{2k + 1}$ divides the $|G|$.
 \item Let $\beta$ be odd. Then, $D$ is a GSHDS in $G$ if and only if $A_{G,G_1} d = \sqrt{\frac{|G|}{p}} \overline{d}$, where $\overline{d}$ is the $\pm 1$ representation of the dual GSHDS of $D$.
\end{enumerate}
\end{prop}

Using the previous proposition, we extend the known exponent bounds for skew Hadamard difference sets.

\section{Background on $S$-rings}

In our study of the $A_{G,G_1}$ incidence structures, we will make use of the 
character tables of the general association schemes given by $\Z[G]^K$. 

\begin{df}
Let $G$ be an abelian group, and $K \subset Aut(G)$ a subgroup of automorphisms of $G$. Then, the set: 
\begin{eqnarray}
\nonumber
 \Z[G]^K &=& \{ a \in \Z[G] \mid \text{ $a = \sum_{g \in G} a_g x^g$ and  $ a_{\sigma(g)} = a_{g}$  for all $ \sigma \in K$ } \}
\end{eqnarray}
is called the $K$-invariant association scheme of $G$ denoted by $H(G,K)$.
\end{df}

The association schemes $H(G,K)$ are examples of $S$-rings introduced by Bannai and Ito in \cite{Bannai1}. Their character tables can be calculated using the standard representation theory of abelian groups.

\begin{prop}
\label{hgkP1}
 Let $H \subset Aut(G)$ and $G$ an abelian group. Let $X = \{ O_{g_0}, \ldots , O_{g_r} \}$ be the set of $H$-orbits on $G$, where $g_0 = 0$ is the identity element of $G$. Let $\chi_i = \theta(g_i)$, and let $Y = \{Q_{\chi_0},\ldots,O_{\chi_r} \}$ be the corresponding set of $H$-orbits on $\overline{G}$. Define
 $e_{O_{\chi_i}} = \sum_{\chi \in O_{\chi_i} } e_{\chi}$, where $e_{\chi} = \sum_{g\in G} \frac{1}{|G|} \overline{\chi(g)}x^g$. Then, 
\begin{enumerate}
 \item The set $X$ is a basis of
  $H(G,K)$.
 \item The set $Z = \{ e_{\chi_0}, \ldots ,e_{\chi_r}\}$ is the basis of 
   primitive idempotents  of $H(G,K)$.
 \item We have $O_{g_i}(x) = \sum_{j=0}^r \chi_j(O_{g_i}) e_{O_{\chi_i}}$, where 
$\chi_j(O_{g_i}) = \sum_{g \in O_{g_i}} \chi_j(g)$.
 \item We have $e_{O_{\chi_i}} = \frac{1}{|G|} 
    \sum_{j=0}^r  \overline{O_{\chi_i}(g_j)} O_{g_j}(x)$, where 
    $O_{\chi_i}(g_j) = \sum_{\chi \in O_{\chi_i}} \chi(g_j)$.
 \item Let $C_{G,K}$ be the $(r+1) \times (r+1)$ matrix given by
  $C_{G,K}(\chi_i, g_j) = \chi_i(O_{g_j})$, and $B_{G,K}$ be the
 $(r+1) \times (r+1)$ matrix given by 
 $B_{G,K}(g_i,\chi_j) = \overline{O_{\chi_j}(g_i)}$. Then,
 \begin{eqnarray}
\nonumber
  B_{G,K} C_{G,K} &=& |G| I_{r+1,r+1}, \\
\nonumber
  C_{G,K} B_{G,K} &=& |G| I_{r+1,r+1}, \\
\nonumber
  B_{G,K} &=& \overline{C_{G,K}},
 \end{eqnarray}
 where $I_{r+1,r+1}$ is the $(r+1) \times (r+1)$ identity matrix.
\end{enumerate}
\end{prop}

\begin{proof}
It suffices to prove $B_{G,K} = \overline{C_{G,K}}$. The rest of the claims follow from the standard representation theory of abelian groups.
Note that,
\begin{eqnarray}
 \nonumber
 C_{G,K}(\chi_i, g_j) &=& \chi_i(O_{g_j}) \\
 \nonumber
   &=& \theta(g_i, O_{g_j}) \\
 \nonumber
   &=& \theta(O_{g_j}, g_i) \\
 \nonumber
   &=& O_{\chi_j}(g_i)\\
 \nonumber
   &=& \overline{B_{G,K}(g_i,\chi_j)}
\end{eqnarray}
\end{proof}

\section{Results from the $H(G,G_1)$ and $H(G,G_2)$ Association Schemes}

For the rest of the paper, we will assume a noncanonical isomorphism
$\theta : G \rightarrow \overline{G}$ such that
$\theta(g)(g') = \theta(g')(g)$. We will also assume that $n_0$ is a fixed non-quadratic residue mod $exp(G)$. We will denote by $G_1 = (\Z/ exp(G)\Z)^*$ and $G_2 = (\Z/ exp(G)\Z)^{*2}$. We will let $g_0, g_1, \ldots, g_r$ be the
$G_1$-orbit representatives of the action of $G_1$ on $G$, where $g_0 = 0$ is the identity element of $G$.
From the representatives $g_0,\ldots, g_r$ and $\theta$, we will let  
$\chi_i =\theta(g_i)$, which will be the corresponding $G_1$-orbit 
representatives
of the action of $G_1$ on $\overline{G}$. Also, we will denote by $\Omega_g$ the $G_1$ orbit of $g$, and by $O_g$ the $G_2$ orbit of $g$. Clearly, $\Omega_g = O_g + O_g^{(n_0)}$.

The following lemma will be used in our study of QRSs.

\begin{lem}
\label{lem1}
Let $G$ be an abelian group. Then,
\begin{enumerate}
 \item We have,
\begin{eqnarray}
\nonumber
 [|\Omega_{g_0}|, \ldots, |\Omega_{g_r}| ] C_{G,G_1} &=& |G| [ 1, 0, \ldots, 0]^t
\end{eqnarray}
  \item Using the basis $\{ O_{g_0}, O_{g_1}, \ldots, O_{g_r}, O_{g_1}^{(n_0)}, \ldots, O_{g_r}^{(n_0)} \}$ of $H(G,G_2)$, the matrix $C_{G,G_2}$ has decomposition,
\begin{eqnarray}
 \nonumber
  C_{G,G_2} &=& \left[ 
		\begin{array}{ccc}
		 1& j_{G_1} & j_{G_1} \\
		 j & A_0 & A_0^{(n_0)} \\
		 j & A_0^{(n_0)} & A_0
		\end{array}
	\right]
\end{eqnarray}
 Where $j$ is the vector of all $1$s of size $r \times 1$, $j_{G_1} = [ |O_{g_1}|, \ldots, |O_{g_r}| ]$ of size $1 \times r$, and
\begin{eqnarray}
 \nonumber
 A_0(O_{\chi_i},O_{g_j}) &=& \chi_i(O_{g_j}) \\
 \nonumber
 A_0^{(n_0)}(O_{\chi_i},O_{g_j}) &=& \chi_i(O_{g_j}^{(n_0)})
\end{eqnarray}
\item We must have,
\begin{eqnarray}
 \nonumber
  (\overline{A_0} - \overline{A_0^{(n_0)}}) (A_0 - A_0^{(n_0)}) &=& |G| I
\end{eqnarray}
\end{enumerate}
\end{lem}

\begin{proof}
Clearly, by Proposition (\ref{hgkP1}),
\begin{eqnarray}
\label{eq1_p1}
\overline{C_{G,G_1}} C_{G,G_1} &=& |G| I
\end{eqnarray}

Note that the first row of $\overline{C_{G,G_1}}$ is given by
$[|\Omega_{g_0}|, \ldots, |\Omega_{g_r}| ]$. Thus, Part 1 follows Equation (\ref{eq1_p1}).
Part 2 is a clear calculation of $C_{G,G_2}$ using the basis $\{ O_{g_0}, O_{g_1}, \ldots, O_{g_r}, O_{g_1}^{(n_0)}, \ldots, O_{g_r}^{(n_0)} \}$. 

Clearly, by Proposition (\ref{hgkP1}), $|G|I$ equals,
\begin{eqnarray}
 \nonumber
 \left[ 
	\begin{array}{ccc}
	1 & j_{G_1} & j_{G_1} \\
	j & \overline{A_0} & \overline{A_0^{(n_0)}} \\
	j & \overline{A_0^{(n_0)}} & \overline{A_0} 
	\end{array}
 \right]
 \left[ 
	\begin{array}{ccc}
	1 & j_{G_1} & j_{G_1} \\
	j & A_0 & A_0^{(n_0)} \\
	j & A_0^{(n_0)} & A_0 
	\end{array}
 \right]
\end{eqnarray}
Thus, we deduce the following equations,
\begin{eqnarray}
 \nonumber
  J_{G_1} + \overline{A_0} A_0 + \overline{A_0^{(n_0)}} A_0^{(n_0)} &=& |G| I\\
 \nonumber
  J_{G_1} + \overline{A_0^{(n_0)}} A_0 + \overline{A_0} A_0^{(n_0)} &=&  0,
\end{eqnarray}
where $J_{G_1} = J Diag(|O_{g_1}|, \ldots, |O_{g_r}|)$. From both of these equations, we can deduce Part 3.
\end{proof}

We will calculate the matrix $A_0$ whenever $G$ is an abelian $p$-group. However, we need some auxilary lemmas before we do so.

\begin{lem}
\label{lem2}
Define,
\begin{eqnarray}
 \nonumber
 \omega &=& \sum \{ \eta_p^n \mid 1 \leq n \leq p-1, \QRSym{n} = 1 \} \\
 \nonumber
 \omega^{(n_0)} &=& \sum \{ \eta_p^n \mid 1 \leq n \leq p-1, \QRSym{n} =  -1 \} 
\end{eqnarray}
where $\eta_p$ be a primitive $p$-th root of unity. Then, we can choose $\eta_p$ so that,
\begin{eqnarray}
 \nonumber
 \omega &=& \frac{-1 + \sqrt{\QRSym{-1}p}}{2}\\
 \nonumber
 \omega^{(n_0)} &=& \frac{-1 - \sqrt{\QRSym{-1}p}}{2}
\end{eqnarray}
\end{lem}

\begin{proof}
 Consider the association scheme $H((\Z/ p\Z), (\Z/ p\Z)^{*2})$. Clearly, this is a $3$-dimensional algebra with basis $\{ O_0, O_1, O_{n_0} \}$. Also,
note that $D = O_1$ is a GSHDS. Thus, by Proposition (\ref{PropCharValue}), we have for an arbitrary non-principal character $\chi$ of $(\Z/ p\Z)$.
\begin{eqnarray}
 \nonumber
 \chi(D) &=& \sum\{ \eta_p^{m n} \mid 1 \leq n \leq p-1, \QRSym{n} = 1 \} \\
  \nonumber
   &=& \frac{ -1 + \epsilon_{\chi} \sqrt{\QRSym{-1} p}}{2}
\end{eqnarray}
where $\chi(1) = \eta_p^m$. Clearly, $\{ \chi(D) , \chi(D^{(n_0)}) \} = \{ \omega, \omega^{(n_0)} \}$. Thus, the result follows.
\end{proof}

From now on, we will assume $\eta_p$ is a primitive $p$-th root of unity so that the conclusion of Lemma (\ref{lem2}) is satisfied.

\begin{lem}
\label{lem3}
Let $O_{g_i}$ be a $G_2$ orbit of $G \backslash \{ 0 \}$, and $\chi_j$ a $G_2$ orbit representative of $\overline{G} \backslash \{ \chi_0 \}$ . Then,
\begin{eqnarray}
 \nonumber
 \chi_j(O_{g_i}(x) - O_{g_i}(x^{n_0})) &=& A_{G,G_1}( \Omega_{\chi_j}, \Omega_{g_i}) \sqrt{\QRSym{-1} p }
\end{eqnarray}
\end{lem}

\begin{proof}
Clearly, $O_{g_i} = \LRGen{p\cdot g_i} \sum \{ [n \cdot g_i] \mid 1 \leq n \leq p-1, \QRSym{n} = 1 \}$. Thus,
\begin{eqnarray}
\nonumber 
\chi_j(O_{g_i}(x)) &=& \left\{ \begin{array}{cc}
	o(p\cdot g_i) \frac{p-1}{2} & \text{ if } \chi_j(g_i) = 1\\
	\omega o(p \cdot g_i) & \text{ if $\chi_j(g_i) = \eta_p^n$ where $\QRSym{n} =1$ and $1 \leq n \leq p -1$}\\ 
	\omega^{(n_0)} o(p \cdot g_i) & \text{ if $\chi_j(g_i) = \eta_p^n$ where $\QRSym{n} = -1$ and $1 \leq n \leq p -1$}\\ 
	0 & \text{ else }
\end{array}
\right.
\end{eqnarray}
Thus, $\chi_j(O_{g_i}(x) - O_{g_i}(x^{n_0})) $ equals

\begin{eqnarray}
\nonumber 
\left\{ \begin{array}{cc}
	0 & \text{ if } \chi_j(g_i) = 1\\
	(\omega - \omega^{(n_0)}) o(p \cdot g_i) & \text{ if $\chi_j(g_i) = \eta_p^n$ where $\QRSym{n} =1$ and $1 \leq n \leq p -1$}\\ 
	(\omega^{(n_0)} - \omega) o(p \cdot g_i) & \text{ if $\chi_j(g_i) = \eta_p^n$ where $\QRSym{n} = -1$ and $1 \leq n \leq p -1$}\\ 
	0 & \text{ else }
\end{array}
\right.
\end{eqnarray}
Note that by choice of $\eta_p$ in Lemma (\ref{lem2}), we have that $\omega - \omega^{(n_0)} = \sqrt{\QRSym{-1} p }$, and $\omega^{(n_0)} - \omega = - \sqrt{\QRSym{-1}p}$. Thus, the result follows.

\end{proof}

\begin{cor}
\label{cor1}
 Let $G$ be an abelian $p$-group. Then,
 \begin{eqnarray}
\nonumber
  A_{G,G_1}^2 &=& \frac{|G|}{p} I_{r,r},
 \end{eqnarray} 
 where $r$ is the number of $G_1$ orbits in $G\backslash \{ 0 \}$, 
 $I_{r,r}$ is the identity $r\times r$ matrix, and
$r= r(G)$ is the number of $G_1$ orbits of 
$\widetilde{G} = G\backslash \{ 0 \}$.
\end{cor}

\begin{proof}
By Lemma (\ref{lem3}), we have that $A_0 - A_0^{(n_0)} = \sqrt{\QRSym{-1} p} A_{G,G_1}$. Also, by Lemma (\ref{lem1}) part 3,
\begin{eqnarray}
 \nonumber
 |G|I &=& \overline{(A_0 - A_0^{(n_0)})}(A_0 - A_0^{(n_0)})\\
 \nonumber
 &=& \overline{\sqrt{\QRSym{-1}p}} \sqrt{\QRSym{-1}p} A_{G,G_1}^2\\
 \nonumber
 &=& p A_{G,G_1}^2
\end{eqnarray}
From which, the result follows.
\end{proof}

We calculate $A_{G,G_1}$ for $G=(\Z/p^s\Z)$ as an illustrative example and leave
the proof as an exercise.

\begin{prop}
\label{HeckeG2P1}
Let $G=(\Z/ p^s\Z)$. Then,
\begin{enumerate}
 \item A full set of $G_1$ orbit representatives of $G \backslash \{ 0 \}$ is given 
by $X_1 = \{ 1,p, p^2, \cdots, p^{s-1}\}$.
 \item A commutative pairing for $G$ is given by:
 \begin{eqnarray}
   \nonumber
   \theta(a)(b) &=& \eta_{p^s}^{ab},
 \end{eqnarray}
 \item By using,
 \begin{eqnarray}
 \nonumber
   X &=& \C\Omega_{p^{s-1}} \oplus \C \Omega_{p^{s-2}} \oplus \cdots \oplus \Omega_{1},\\
 \nonumber
   Y &=& \C\Omega_{\chi_{p^{s-1}}} \oplus \C \Omega_{\chi_{p^{s-2}}} \oplus \cdots \oplus \Omega_{\chi_{1}},
 \end{eqnarray}
 where $\chi_{p^i} = \theta(p_i)$, we can calculate $A_{G,G_1}$ as:
 \begin{eqnarray}
 \nonumber
   A_{G,G_1} &=& \left( \begin{array}{cccccc}
	 0& 0&0&\cdots&0&p^{s-1}\\
	0&0&0&\cdots&p^{s-2}&0\\
	\vdots&\vdots&\vdots&\cdots&\vdots&\vdots\\
	0&0&p^2&\cdots&0&0\\
	0&p&0&\cdots&0&0\\
	1&0&0&\cdots&0&0	
	\end{array}
	\right).
 \end{eqnarray}
\end{enumerate}
\end{prop}

We close this section by noting that $A_{G,G_1}$ is an $Aut(G)$ map. Clearly, $A_{G,G_1}$ is map defined on the set
$Y \times X$, where $X = \{ \Omega_{g_1}, \ldots, \Omega_{g_r}\}$ and $ Y = \{ \Omega_{\chi_1}, \ldots, \Omega_{\chi_r} \}$. 
We define an action of $Aut(G)$ on $X$ by,
\begin{eqnarray}
 \nonumber
 \rho_X(\sigma) (\Omega_{g_i}) &=& \QRSym{n} \Omega_{g_j} ,
\end{eqnarray}
where $\sigma(g_i) = n \cdot g_j$. Clearly, this action is a representation of $Aut(G)$ on $GL(\bigoplus \{ \C \Omega_{g_i} \mid i = 1, \ldots, r \} )$.
Similary, we define the action of $Aut(G)$ on $Y$ by,
\begin{eqnarray}
 \nonumber
 \rho_{Y}(\sigma) (\Omega_{\chi_i}) &=& \QRSym{n} \Omega_{\chi_j},
\end{eqnarray}
where $\sigma \cdot \chi_i = n \cdot \chi_j$; and for general $\sigma \in Aut(G)$, we define $\sigma \cdot \chi = \chi \circ \sigma^{-1}$. Clearly,
this action is a representation of $Aut(G)$ on $GL(\bigoplus \{ \C \Omega_{\chi_i} \mid i = 1, \ldots, r \} )$.

We leave the next proposition as an exercise to the reader. 
\begin{prop}
\label{prop2_1}
 Let $\sigma \in Aut(G)$. Then,
\begin{enumerate}
 \item There is a unique $\gamma \in Aut(G)$ such that for all $g \in G$, $\theta(\gamma(g))(g') = \theta(g,\sigma(g'))$. We will denote by $\gamma$ by $\sigma^*$. 
 \item The map $\sigma \rightarrow \sigma^*$ is an involution of the $Aut(G)$ that commutes with the map $\sigma \rightarrow \sigma^{-1}$.
 \item The matrix $A_{G,G_1}$ is an $Aut(G)$-map. That is, $\rho_{Y}(\sigma) A_{G,G_1} = A_{G,G_1} \rho_{X}(\sigma)$.
 \item The following holds, $\rho_{Y} (\sigma) = \rho_X( (\sigma^*)^{-1} )$.
\end{enumerate}
\end{prop}

\section{Results for Quadratic Residue Slices}

Using a similar technique as Chen-Sehgal-Xiang in \cite{Xia1}, we calculate the character value of a general QRS.

\begin{prop}
\label{prop5}
Let $D$ be a QRS in $G$. Then,
\begin{eqnarray}
\nonumber
 \chi(D) &=& \frac{-1 + (2 a_{\chi} + 1)\sqrt{\QRSym{-1}p} }{2},
\end{eqnarray}
where $a_{\chi} \in \Z$. We will call the number $2 a_{\chi} + 1$ the Difference Coefficient of $D$ in $G$ evaluated at $\chi$, and we will denote this number by
$d_G(\chi, D)$. 
\end{prop}

\begin{proof}
Clearly, $\chi(D) \in \Q(\eta_{p^s})$ where $exp(G) = p^s$. Also, $\chi(D)$ is an algebraic integer that is fixed under the action of the quadratic residues mod $p^s$. That is,
$\chi(D) \in Fix(G_2)$ the fixed field of $G_2$. It can be shown that $Fix(G_2) = \Q(\omega)$, and that the ring of integers of $\Q(\omega)$ is given by $\Z \oplus \Z \omega$. Thus,
\begin{eqnarray}
 \nonumber
 \chi(D) &=& b_{\chi} + a_{\chi}\omega
\end{eqnarray}
Because $n_0$ is the non-trivial Galois automorphism of $\Q(\omega)$, we must also have:
\begin{eqnarray}
 \nonumber
 \chi(D^{(n_0)}) &=& a_{\chi} + b_{\chi}\omega^{(n_0)}
\end{eqnarray}
Because $D(x) + D(x^{n_0}) = G(x) - [1]$ and $\omega + \omega^{(n_0)} = -1$, we must have that $2 a_{\chi} - b_{\chi} = -1$. Thus, $b_{\chi} = 2 a_{\chi}  + 1$ and,
\begin{eqnarray}
\nonumber
 \chi(D) &=& a_{\chi} + (2 a_{\chi} + 1) \omega \\
\nonumber
 &=& \frac{-1 + (2 a_{\chi} + 1)\sqrt{\QRSym{-1}p} }{2}.
\end{eqnarray}
Therefore, the result follows.
\end{proof}

We introduce a $\pm 1$ representation of a QRS $D$.
\begin{df}
Let $D$ be a QRS in $G$. Let $\{ O_{g_1}, \ldots , O_{g_r}, O_{g_1}^{(n_0)}, \ldots, O_{g_r}^{(n_0)} \}$ be the $G_2$ orbits of $G \backslash \{ 0 \}$. Then,
\begin{eqnarray}
 \nonumber
 D(x) - D(x^{n_0}) &=& d_1 (O_{g_1}(x) - O_{g_1}(x^{n_0})) + \ldots + d_r (O_{g_r}(x) - O_{g_r}(x^{n_0}) ),
\end{eqnarray}
where,
\begin{eqnarray}
 \nonumber
 d_i &=& \left\{ \begin{array}{cc}
	1 & \text{ if $g_i \in D$ },\\
	-1 & \text{ if $g_i \in D^{(n_0)}$ }.
	\end{array} \right.
\end{eqnarray}
We will call the vector $d=[d_1, \ldots, d_r]^t$ the $\pm 1$s representation of $D$.
\end{df}

\begin{cor}
\label{cor_qrs_1}
Let $D$ be a QRS in $G$ and define $df(D) = [ d_G(\chi_1,D), \ldots, d_G(\chi_r,D)]^t$. Then, $A_{G,G_1} d = df(D)$.
\end{cor}

\begin{proof}
 Clearly, by Proposition (\ref{prop5}), 
\begin{eqnarray}
 \nonumber
 \chi_i(D(x) - D(x^{n_0})) &=& d_G(\chi_i, D) \sqrt{\QRSym{-1}p}.
\end{eqnarray}
Also, note that,
\begin{eqnarray}
 \nonumber
 \chi_i(D(x)-D(x^{n_0})) &=& \sum_{j=1}^r d_i \chi_i (O_{g_j}(x) - O_{g_j}(x^{n_0}))\\
 \nonumber
 &=& \sum_{j=1}^r d_j A_{G,G_1}(\Omega_{\chi_i}, \Omega_{g_j}) \sqrt{\QRSym{-1}p}
\end{eqnarray}
Where we have used Lemma (\ref{lem3}). Clearly, the result follows from the previous equations.
\end{proof}

We proceed to show the $p$ divisibility conditions of the Difference Coefficients of a QRS $D$.

\begin{prop}
\label{prop6}
Let $D$ be a QRS in $G$, where $|G| = p^{\beta}$ and $exp(G) = p^s$. Then,
\begin{enumerate}
 \item Let $p^k$ divide $d_G(\chi_i, D)$ for $i=1,\ldots, r$. Then, $2k +1 \leq \beta$.
 \item Assume $\beta=2\alpha +1$ is odd. The set $D$ is a GSHDS if and only if $p^{\alpha}$ divides $d_G(\chi_i,D)$ for $i=1,\ldots,r$.
\end{enumerate}
\end{prop}

\begin{proof}
Let $\{ \Omega_{g_0}, \Omega_{g_1}, \ldots, \Omega_{g_r} \}$ be a basis for $H(G,G_1)$ where $g_0 = 0$. 
Consider the element $d(D)(x) = (D(x) - D(x^{n_0}))( D(x) - D(x^{n_0}))$. Clearly, $d(D) \in H(G,G_1)$. Thus, 
\begin{eqnarray}
 \nonumber
 d(D)(x) &=& \sum_{i=0}^r a_{g_i} \Omega_{g_i}(x)
\end{eqnarray} 
Where $a_{g_i} \in \Z$.
Note that for
$\chi_i = \theta(g_i)$, we have
\begin{eqnarray}
 \nonumber
 \chi_i(d(D)) &=& (d_G(\chi_i,D) \sqrt{\QRSym{-1} p} )^2 \\
\nonumber
 &=& \QRSym{-1} p d_G(\chi_i,D)^2.
\end{eqnarray}
Let $a(D) = [0,d_G(\chi_1,D)^2,\ldots, d_G(\chi_r,D)^2]^t$. Clearly,
 \begin{eqnarray}
\nonumber
  C_{G,G_1} \overline{d(G)} &=& p \QRSym{-1} a(D)
 \end{eqnarray}
Where $\overline{d(G)} = [ a_{g_0}, \ldots, a_{g_r}]^t$. Multiply the previous equation by $j_{G_1} = [|\Omega_{g_0}|, \ldots, |\Omega_{g_r}|]^t$. By Lemma (\ref{lem1}), we have that
\begin{eqnarray}
 \nonumber
 |G| a_{g_0} &=& |G| \LRGen{[1,0,\ldots,0]^t,\overline{d(G)} } \\
 \nonumber
 &=& p \QRSym{-1} \LRGen{j_{G_1}, a(D)} 
\end{eqnarray}
A direct calculation yields $a_{g_0} = d(D)(0) = \QRSym{-1} ( |G| - 1)$. Also, by assumption, $a(D) = p^{2k} b$ for some vector $b$ of positive integers. Thus,
\begin{eqnarray}
\label{eqn3_1}
 |G| \QRSym{-1} ( |G| -1 ) &=& p \QRSym{-1} p^{2k} \LRGen{j_{G_1}, b}
\end{eqnarray}
Hence, it follows that $p^{2k +1}$ divides $|G|$. Thus, part 1 follows.

Now we show part 2. Assume that $\beta = 2\alpha + 1$ and that $D$ is a GSHDS. By Proposition(\ref{PropCharValue}), $d_G(\chi_i,D)  = p^{\alpha} \epsilon_{\chi_i}$ where $\epsilon_{\chi_i} \in \{ -1, 1\}$.
Clearly, $p^{\alpha}$ divides $d_G(\chi_i,D)$ for $i=1,\ldots,r$. 

Now, assume that $p^{\alpha}$ divides $d_G(\chi_i,D)$ for $i=1,\ldots,r$. It suffices to show that $d_G(\chi_i,D) = \pm p^{\alpha}$, because then it will follow that for arbitrary nonprincipal character $\chi$, 
\begin{eqnarray}
\nonumber
 \chi(D) \chi(D^{(n_0)}) &=& \frac{ 1 - \QRSym{-1} v }{4}
\end{eqnarray}
Thus, by Fourier inversion,
\begin{eqnarray}
 \nonumber
  D(x) D(x^{n_0}) &=& k_0 [1] + \lambda D(x) + \lambda D(x^{n_0}),
\end{eqnarray}
where $k_0 = \frac{ 1 - \QRSym{-1}}{2} \frac{v-1}{2}$ and $\lambda = \frac{ v -2 + \QRSym{-1}}{4}$. Which implies that $D$ is a GSHDS.

We show that $d_G(\chi_i,D) = \pm p^{\alpha}$. Consider Equation (\ref{eqn3_1}). Since $k = \alpha$, we must have,
\begin{eqnarray}
\nonumber
 p^{2\alpha +1 } \QRSym{-1} (|G| - 1) &=& p^{2\alpha + 1} \QRSym{-1} \LRGen{j_{G_1}, b}
\end{eqnarray}
Since $\LRGen{j_{G_1}, j} = |G| -1$, we must have $ \LRGen{j_{G_1}, j} = \LRGen{j_{G_1},b}$. This forces $b= j$ because $b$ is a vector of positive integers. In particular, this shows that $d_G(\chi_i,D)^2 = p^{2\alpha}$. The result follows.
\end{proof}

Proposition (\ref{prop6}) motivates the definition of the $p$-divisibility $\nu_p(D)$ of a general QRS $D$.
\begin{df}
Let $D$ be a QRS in $G$, and $d$ the corresponding representation of $\pm 1$s. Then, the $p$-divisibility $\nu_p(D)$ of $D$ is the integer $k$ such that
$p^k$ divides all entries of the vector $A_{G,G_1} d$. 
\end{df}

\begin{cor}
\label{qrs_cor2}
Let $G$ be an abelian group of order $p^{\beta}$ and $D$ a QRS in $G$. Then,
\begin{enumerate}
 \item If $\beta = 2\alpha$, then $v_p(D) \leq \alpha - 1$. 
 \item If $\beta = 2\alpha + 1$, then $v_p(D) \leq \alpha$, and $v_p(D) = \alpha$ if and only if $D$ is a GSHDS.
\end{enumerate}
\end{cor}

We leave as an open problem the characterization of QRSs $D$ in $G$ with maximal value of $v_p(D)$ whenever $|G| = p^{2\alpha}$. 
The next proposition was originally formulated in \cite{Xia1} in the context of SHDSs to prove the Chen-Sehgal-Xiang exponent bound.

\begin{prop}
\label{prop7}
Let $G$ be an abelian $p$-group with $exp(G) = p^s$. Define $G_l = \{ g \in G \mid p^l \cdot g = 0 \}$. Let $\chi$ be a non-principal character of $G$. Then,
\begin{eqnarray}
 \nonumber
 d_G(\chi,D) &=& d_{G_l}(\chi\mid_{G_l}, G_l) \text{ mod } p^l
\end{eqnarray}
\end{prop}

\begin{proof}
Without loss of generality, using the notation of Definition (\ref{AGG1Def}), we can assume that $\chi = \chi_i$ for some $i$ because $d_G(\chi,D) = \pm d_G(\chi_i,D)$ for some $\chi_i$.
Consider,
\begin{eqnarray}
\nonumber
  d_G(\chi_i,D) &=& \sum_{j=1}^r A_{G,G_1}(\Omega_{\chi_i}, \Omega_{g_j}) d_j \\
\nonumber
	&=& \sum \{ o(p\cdot g_j) \QRSym{n} d_j \mid \chi_i(g_j) = \eta_p^n, 1 \leq n \leq p -1 \} \\
\nonumber 
 &=& \sum \{ o(p\cdot g_j) \QRSym{n} d_j \mid \chi_i(g_j) = \eta_p^n, 1 \leq n \leq p -1, o(g_j) \leq p^l \} \text{ mod } p^l \\
\nonumber
 &=& d_{G_l}(\chi_i\mid_{G_l}, G_l) \text{ mod } p^l
\end{eqnarray}
\end{proof}

We close this section with a formula relating the Difference Coefficients to the Difference Intersection Numbers.


\begin{prop}
\label{prop8}
Let $G$ be an abelian $p$-group, and $L \subset G$ a subgroup with $H = \frac{G}{L}$. Let
$\pi : G \rightarrow H$ be the canonical projection.
Choose the restriction of $\theta$ to $H$ as the commutative pairing $\theta_H = \theta\mid_H : H \rightarrow \overline{H}$. 
Choose a set of $H_1$ orbit representatives $\{h_1,\ldots, h_s\}$  of $H \backslash \{ 0 \}$, and choose
 $\{ \chi_1 = \theta_H(h_1), \ldots, \chi_s = \theta_H(h_s) \}$ as the corresponding set of $H_1$ orbit
representatives of $\overline{H} \backslash \{ \chi_0\}$.
Let $g_i$ be any lift of $h_i$ under $\pi$, and let $\chi_i' = \chi_i \circ \pi$ be the extension of $\chi_i$ to $G$. Define
$A_{H,H_1}$ using the choices for $\chi_i$, $h_i$, and $\theta_H$. Then,

\begin{enumerate}
 \item The following holds, 
\begin{eqnarray}
  \nonumber
 \pi(D - D^{(n_0)}) &=& \sum_{i=1}^s \nu_{G,L}(g_i,D) (O_{h_i}   - O_{h_i}^{(n_0)})
\end{eqnarray}
 \item The following holds
 \begin{eqnarray}
  \nonumber
  A_{H,H_1} \LRVec{ \begin{array}{c} \nu_{G,L}(g_1,D)\\ \vdots \\ \nu_{G,L}(g_s,D) \end{array} } &=&
 \LRVec{\begin{array}{c} d_G(\chi_1',D) \\ \vdots \\ d_G(\chi_s',D) \end{array} }
 \end{eqnarray}
\end{enumerate}
\end{prop}

\begin{proof}
Let $G = l_1 L \cup \cdots \cup l_t L$ and $m_1 = \pi(l_1), \cdots, m_t = \pi(l_t)$, where $t = |H|$. Clearly, 
\begin{eqnarray}
\nonumber
 D &=& \sum_{i=1}^t D \cap l_i L \\
\nonumber
 D^{(n_0)} &=& \sum_{i=1}^t  D^{(n_0)} \cap l_i L
\end{eqnarray}
Thus,
\begin{eqnarray}
 \nonumber
  \pi( D - D^{(n_0)}) &=& \sum_{i=1}^t \nu_{G,L}(l_i,D) [m_i]
\end{eqnarray}
Note that $\nu_{G,L}( n \cdot g, D) = \QRSym{n} \nu_{G,L}(g, D)$. Thus,
\begin{eqnarray}
 \nonumber
  \pi( D - D^{(n_0)}) &=& \sum_{j=1}^s \nu_{G,L}(g_j,D) (O_{h_j} - O_{h_j}^{(n_0)})
\end{eqnarray} 
Hence, part 1 follows.

Apply $\chi_i$ to part 1. Clearly,

\begin{eqnarray}
\nonumber
 \chi_i \circ \pi ( D - D^{(n_0)}) &=& \sum_{j=1}^s \nu_{G,L}(g_j,D) \chi_i( O_{h_j} - O_{h_j}^{(n_0)})
\end{eqnarray}

By Proposition (\ref{prop5}) and Lemma (\ref{lem3}), we have,
\begin{eqnarray}
\nonumber 
 \chi_i \circ \pi (D - D^{(n_0)}) &=&  d_G(\chi_i',D) \sqrt{\QRSym{-1}p}\\
\nonumber
  \sum_{j=1}^s \nu_{G,L}(g_j,D) \chi_i( O_{h_j} - O_{h_j}^{(n_0)}) &=& \sum_{j=1}^s \nu_{G,L}(g_j,D) A_{H,H_1}(\Omega_{\chi_i}, \Omega_{h_j}) \sqrt{\QRSym{-1}p}
\end{eqnarray}
Clearly, the part 2 follows from the previous two equations.
\end{proof}

\section{Exponent Bounds}

We show the exponent bounds known for SHDSs by making use of the following Lemma.

\begin{lem}
\label{lem_exp1}
The following hold,
\begin{enumerate}
 \item Let $G= (\Z/ p^s\Z)$ and $D$ a QRS in $G$. Then, $\nu_p(D) = 0$.
 \item Let $G= (\Z/ p^s\Z) \times (\Z/ p^s\Z)$ and $D$ a QRS in $G$. Then, $\nu_p(D) = 0$.
\end{enumerate}
\end{lem}

\begin{proof}
We show part 1. Let $d$ be the $\pm 1$s representation of $D$. Clearly, by Proposition(\ref{HeckeG2P1}), there is a difference coefficient of $D$ that is $\pm 1$. Hence, $\nu_p(D) = 0$.

Now we show Part 2. Assume otherwise. That is, there is a QRS $D$ such that $\nu_p(D) \geq 1$. Let $D' = D \cap H$ where $H = \{ g \in G \mid p \cdot g = 0 \}$. By Proposition (\ref{prop7}), $\nu_p(D') \geq 1$. 
Note that $H = (\Z/ p\Z) \times (\Z/ p\Z)$, hence by Corollary (\ref{qrs_cor2}), $\nu_p(D') = 0$. Clearly a contradiction. 
\end{proof}

Now, we show Johnsen's exponent bound.

\begin{prop}
Let $G$ admit a GSHDS $D$. Assume that $|G| = p^{2\alpha +1 }$ and that $exp(G) = p^s$. Then, $s \leq \alpha + 1$.
\end{prop}

\begin{proof}
Clearly, $G = (\Z/ p^s\Z) \times L$, for some subgroup $L$. Hence, there is a projection $\pi : G \rightarrow H = (\Z/ p^s\Z)$. We proceed to apply Proposition(\ref{prop8}). Hence, there is an integral vector $\nu$ such that,
\begin{eqnarray}
\nonumber
 A_{H,H_1} \nu &=& df_H(D)
\end{eqnarray}
where $df_H(D)$ are the difference coefficients of $D$ evaluated at the characters of $H$ extended to $G$. Because $D$ is a GSHDS, we must have $df_H(D) = p^{\alpha} d$ where $d$ is a vector of $\pm 1$s. More percisely, $d$ is the $\pm 1$s representation of $\overline{D} \cap \overline{H}$ where $\overline{D}$ is the dual of $D$. Hence,
\begin{eqnarray}
 \nonumber
 A_{H,H_1} \nu &=& p^{\alpha} d
\end{eqnarray}
By applying Corollary (\ref{cor1}), we derive,
\begin{eqnarray}
\nonumber
 p^{s-1 - \alpha } \nu &=& A_{H,H_1} d
\end{eqnarray}
Where $d$ corresponds to a QRS in $\overline{H}$. By Lemma (\ref{lem_exp1}), we must have $s-1 - \alpha \leq 0$. The result follows. 
\end{proof}

Using a similar technique, we show the Chen-Sehgal-Xiang exponent bound.

\begin{prop}
Let $G$ admit a GSHDS $D$. Assume that $|G| = p^{2\alpha +1 }$ and that $exp(G) = p^s$. Then, $2s \leq \alpha + 1$.
\end{prop}

\begin{proof}
Clearly by Proposition(\ref{Camion2}), $G = (\Z/ p^s\Z) \times (\Z/ p^s\Z) \times L$, for some subgroup $L$. Hence, there is a projection $\pi : G \rightarrow H = (\Z/ p^s\Z) \times (\Z/ p^s\Z)$. We proceed to apply Proposition(\ref{prop8}). Hence, there is an integral vector $\nu$ such that,
\begin{eqnarray}
\nonumber
 A_{H,H_1} \nu &=& df_H(D)
\end{eqnarray}
where $df_H(D)$ are the difference coefficients of $D$ evaluated at the characters of $H$ extended to $G$. Because $D$ is a GSHDS, we must have $df_H(D) = p^{\alpha} d$ where $d$ is a vector of $\pm 1$s. More percisely, $d$ is the $\pm 1$s representation of $\overline{D} \cap \overline{H}$ where $\overline{D}$ is the dual of $D$. Hence,
\begin{eqnarray}
 \nonumber
 A_{H,H_1} \nu &=& p^{\alpha} d
\end{eqnarray}
By applying Corollary (\ref{cor1}), we derive,
\begin{eqnarray}
\nonumber
 p^{2s-1 - \alpha } \nu &=& A_{H,H_1} d
\end{eqnarray}
Where $d$ corresponds to a QRS in $\overline{H}$. By Lemma (\ref{lem_exp1}), we must have $2s-1 - \alpha \leq 0$. The result follows. 
\end{proof}

\section{Results From The Galois Rings $GR(p^2,\beta)$}

The derivation of necessary existence conditions for the family of groups $G = (\Z/ p^2\Z)^{2\alpha + 1} \times (\Z/ p\Z)$ will
make use of the Galois Rings $GR(p^2,2\alpha +1)$ viewed as the groups $H=(\Z/ p^2\Z)^{2\alpha +1}$ under the additive operation of the Galois Ring. 
The Galois Rings $GR(p^2,2\alpha +1)$ will also be used to construct the
element $L_0(x)$ of Proposition(\ref{gshds3_prop1}), and to derive a canonical form for the $A_{H,H_1}$ incidence structures.
In this section, we will derive these results.

We will assume the following result that can be shown using standard local algebraic number theory as it is done
in Neukirch's book in \cite{Neu} in the chapter of $p$-adic extensions.

\begin{prop}
\label{SpcCaseP1}
 Let $K=\Q_p(\eta)$ be an umramified extension of $\Q_p$ of
degree $\alpha$, where $\Q_p$ are the $p$-adic numbers. Let $O_K$ be the ring of integers of $K$.
Define $GR(p^k, \alpha)= \frac{O_K}{p^k O_K}$.
Let $q=p^\alpha$, then:
 \begin{enumerate}
 \item The ring $GR(p^k,\alpha) =(\Z/ p^k\Z)(\eta)$ is the Galois extension of the ring
 $(\Z/ p^k\Z)$ of degree $\alpha$ and exponent $p^k$. 
  \item There is a set of ``Teichmuller Units'', $\tau = \{1,\eta,
\ldots, \eta^{q-2} \} \subset (O_K)^*$, of $(q-1)$th primitive roots of unity
in $O_K$
that are distinct mod $p O_K$.
  \item Every element $\gamma \in GR(p^k,\alpha)$ decomposes into:
  \begin{eqnarray}
   \nonumber
    \gamma &=& r_0(\gamma) + r_1(\gamma) p + \cdots + r_{k-1}(\gamma) p^{k-1},
  \end{eqnarray} 
  where the $r_i(\gamma) \in \tau \cup \{ 0 \}$ are unique.
  \item The $Gal(K\mid \Q_p) = \LRGen{Fr}$; where $Fr$, the ``Frobenious'' map, is defined by $Fr(\eta) = \eta^p$.
  \item The $Gal(K\mid \Q_p)$ induces a group of automorphisms of $GR(p^k,\alpha)$
that leaves $(\Z/ p^k\Z)$ invariant and is
  defined by $Fr(\gamma) = r_0(\gamma)^p + r_1(\gamma)^p p + \cdots +
 	r_{k-1}(\gamma)^p p^{k-1}$.
  \item The Trace map $Tr : K \rightarrow \Q_p$ defined by:
  \begin{eqnarray}
   \nonumber
   Tr(x) &=& \sum_{j=0}^{n-1} Fr^i(x),
  \end{eqnarray} 
  induces a Trace map $Tr : GR(p^k,\alpha) \rightarrow  (\Z/ p^k\Z)$.
 \end{enumerate} 
\end{prop}

By using Galois Rings, we give
a list of $H_1$ orbit representatives for $H= (\Z/p^2 \Z)^{\alpha}$
that will prove useful.

\begin{prop}
\label{hzp2a}
Let $H=(\Z/p^2 \Z)^{\alpha}$ and $L= \{ g \in H \mid o(g) = p \} = p \cdot H$.
Then,
\begin{enumerate}
 \item The group $H$ can be viewed as  Galois Ring $GR(p^2,\alpha)$ by using
any isomorphism of additive groups $\kappa :H \rightarrow GR(p^2,\alpha)$.
That is, we can define the corresponding multiplication on $H$ as:
\begin{eqnarray}
 \nonumber
 g*g' &=& \kappa^{-1}(\kappa(g)*\kappa(g') )
\end{eqnarray}
 \item Let $r'$ be the number of orbits of the action of $L_1 = (\Z/p \Z)^*$ 
 on $L\backslash \{ 0 \}$, then  $r' = p^{\alpha -1} + p^{\alpha -2} 
  + \cdots + p + 1$.
 \item  Let $q = p^{\alpha}$. There is a set  
	$\mu_{q-1} = \{ k_1,\ldots , k_{q-1} \}$
	of elements of $H$ of order $p^2$, called the ``Teichmuller Units'', 
	such that:
\begin{eqnarray}
\nonumber
    p \cdot \mu_{q-1} &=& \{ p\cdot k_1,\ldots, p \cdot k_{p^{\alpha} -1} \}\\
\nonumber
		&=& L\backslash \{ 0 \}.
\end{eqnarray}
 \item  There are a sets
	$\{ l_1, \ldots, l_{r'} \} \subset \mu_{q-1}$ 
	and 
	$\{ l_{1}', \ldots ,
	 l_{p^{\alpha -1 }}' \} \subset \mu_{q-1} \cup \{ 0 \}$
	such that,
	\begin{enumerate}
	\item  The elements $\{ p \cdot l_1,\ldots , p \cdot l_{r'}\}$ are orbit 
	representatives of 
	the action of $L_1$ on $L \backslash \{0 \}$.
        \item The elements $\{ p \cdot l_1',\ldots, p \cdot l_{p^{\alpha-1}}' \}$
	are orbit representatives of the action of $((\Z/ p \Z), +)$ on
	$(L,+)$.
	\item The elements $h_{i,j} = l_i + p \cdot l_{i,j}' = l_i*(1+ pl_j')$ 
		, where $i = 1, \ldots, r'$ and 
		$j = 1, \ldots, p^{\alpha -1}$, forms a set of orbit
		representatives of the action of $H_1$ on the elements of order
		$p^2$ in $H$.
	\end{enumerate}
\end{enumerate} 
\end{prop}

\begin{proof}
Part 1 is clear. We show part 2 by observing that $L_1$ has order $p-1$ and is acting on, 
$L\backslash \{ 0 \}$, a set of size $p^{\alpha} -1$. Since this action
has no fixed points, a direct use of burnside's formula for group actions gives
the number of orbits as $\frac{ p^{\alpha} -1}{p-1} = p^{\alpha -1} + 
 p^{\alpha -2} + \cdots + p + 1$. Hence, part 2 follows.

To show parts 3 and 4, we will view $H$ as a Galois Ring; the extra operation of
multiplication will help us construct the desired $H_1$ orbit representatives.

Let $K=\Q_p(\beta)$ be an unramified extension of $\Q_p$,
the $p$-adic numbers,
of degree $\alpha$. Denote by  
 $O_K$ the {\rm ring
of integers in $K$}.
Let $\tau$ be the {\rm Teichmuller Units} of Proposition \ref{SpcCaseP1}.

Define $\mu_{q-1} = \tau$ mod $p^2 O_K$. Since $GR(p^2,\alpha)$ is
a local ring with maximal ideal given by $p GR(p^2,\alpha)$, it follows
that all units have order $p^2$. Hence, every element of
$\mu_{q-1}$ has order $p^2$. To show part 3, let $k_1 \neq k_2$ both in $\mu_{q-1}$ such that
$p \cdot k_1 = p \cdot k_2$ in $L \subset H$. By construction,
there are $\eta_1 \neq \eta_2$ both in $\tau$ such that
$p \cdot \eta_1 = p \cdot \eta_2$ mod $p^2 O_K$. Thus,
$p \cdot (\eta_1 - \eta_2) = 0 $ mod $p^2 O_K$. Since $K$ is an unramified
discrete valuation ring with prime $p$, it follows
that $p$ divides $\eta_1 - \eta_2$, i.e., $\eta_1 = \eta_2$ mod $p O_K$.
This contradicts Proposition \ref{SpcCaseP1} part 2. Hence,
$p\cdot \mu_{q-1} \subset L\backslash \{ 0 \}$ has $q-1$ elements.
Since $|L\backslash \{ 0 \}| = q-1$, it must be the case
that $p \cdot \mu_{q-1} = L \backslash \{ 0 \}$.

Now, we proceed to show part 4. 
By using standard local algebraic number theory, we know
that $K^* = \LRGen{p} \times \mu_{q-1} \times U^{(1)}$; where 
$\LRGen{p} = \{ p^i \mid i \in \Z \}$ and 
$U^{(1)} = \{ x \in O_K \mid x = 1 + p z \text{ where } z \in  O_K \}$.
Thus, we can deduce that $GR(p^2,\alpha)^*$ can be parametrized by tuples in the
following form:
\begin{eqnarray}
\label{rp2a}
GR(p^2,\alpha)^* &=& 
	\{ (a_0,a_1) \mid a_0 \in \mu_{q-1}, a_1 \in \mu_{q-1} \cup \{ 0 \} \},
\end{eqnarray}
where $(a_0,a_1) = a_0(1+ p a_1) \in GR(p^2,\alpha)$. 
We note that the parametrization in Equation (\ref{rp2a}) is more of a 
multiplicative decomposition instead of an additive decomposition. 
That is, $(a_0,0) (0,a_1) = (a_0,a_1)$ but $(a_0, 0) + (0,a_1) \neq (a_0,a_1)$.

Note that $(\Z/p^2 \Z) = \Z_p /p^2 \Z_p \subset O_K/p^2 O_K = GR(p^2,\alpha)$.
A direct calculation shows that:
\begin{eqnarray}
\nonumber
(\Z/p^2 \Z)^* &=& \{ (b_0, b_1) \mid b_0 \in \mu_{p-1}, 
	b_1 \in \mu_{p-1}\cup \{ 0 \} \},
\end{eqnarray}
where $\mu_{p-1}$ is the subset of $\mu_{q-1}$ consisting of the $(p-1)$th
roots of unity.
Note that the action of $(\Z/p^2 \Z)^*$ on $H$ is given by multiplication
of $(\Z/ p^2\Z)$ on $GR(p^2,\alpha)$. Also, note that multiplication of
tuples has a nice formula since:
\begin{eqnarray}
\label{mulrp2a}
 (b_0,b_1) (a_0,a_1)&=& b_0(1+ pb_1)a_0(1+pa_1)\\
\nonumber
	&=& b_0 a_0 (1 + p(b_1 + a_1) ) \text{ mod } p^2 O_K \\
\nonumber
	&=& (b_0 a_0,b_1 + a_1),
\end{eqnarray}
where the addition in the second coordinate is taken as addition in $\F_q$,
and multiplication in the first coordinate is taken as multiplication in
$(\F_q)^*$.

Clearly, if we are considering elements of order $p^2$, then we are considering
the elements of $GR(p^2,\alpha)^*$.
Also, 
the action of $(\Z/p^2 \Z)^* = H_1$ on $GR(p^2,\alpha)^* \subset H$ ``decomposes'' via
the use of Equation (\ref{mulrp2a}) onto the action of $\F_p^*$ on $\F_q^*$ on
the first coordinate and the action of $(\F_p,+)$ on $(\F_q,+)$ on the
second coordinate.

Let $m_1,\ldots,m_{r'}$ be orbit representatives of the action of 
$\mu_{p-1}$ on $\mu_{q-1}$. Also, let $n_1,\ldots,n_{p^{\alpha-1}}$ be
orbit representatives of the action of $(\F_p,+) = \mu_{p-1} \cup \{ 0 \}$
on $(\F_q,+) = \mu_{q-1} \cup \{ 0 \}$. 

We claim that the set of
tuples $(m_i, n_j) = m_i(1+ p n_j)$ form a set of orbit representatives of
the action of $(\Z/p^2 \Z)^{*}$ on $GR(p^2,\alpha)^{*}$. This, will show
part 4.

Let $(m,n) \in GR(p^2,\alpha)^{*}$.
Clearly, $m$ belongs to the orbit of some $m_{i_0}$. That is, there is 
$b_0 \in \mu_{p-1}$ such that $m = b_0 m_{i_0}$. Also, $n$ belongs to the orbit
of some $n_{j_0}$. That is, there is $b_1 \in (\F_p, +) = \mu_{p-1} 
\cup \{ 0 \}$ such that $n = b_1 + n_{j_0}$. Note that this means
$(m,n) = (b_0,b_1) (m_{i_0},n_{j_0})$, where $(b_0,b_1)\in (\Z/p^2 \Z)^*$. 
Hence, the list $(m_i, n_j)$ is 
exahustive.

It suffices to show that each orbit is represented by at most one
$(m_i,n_j)$. Suppose that $(m_i,n_j)$ and $(m_{i'},n_{j'})$ represent
the same orbit $H_1$ orbit of $H \backslash \{ 0 \}$.
Clearly, there is $b \in H_1 = (\Z/ p^2\Z)^*$ such that:
\begin{eqnarray}
 \nonumber
 m_i(1+ p n_j) &=& b* m_{i'}( 1 + p n_{j'}).
\end{eqnarray}

Since, we can decompose $b$ into $(b_0,b_1)$, clearly,
$(m_{i'},n_{j'}) = (b_0,b_1) (m_i,n_j)=(b_0 m_i, b_1 + n_j)$.
Thus $m_{i'} = b_0 m_i$ and $n_{j'} = b_1 + n_j$. Hence, forcing $i = i'$ and
$j= j'$ by the choice of the $m_i$ and $n_j$'s.
\end{proof}

We note that the $l_{i,j}'$'s of Proposition \ref{hzp2a} may
repeat themselves; but, the $l_i$'s and $l_j'$'s do not. In our study of
$GR(p^2,2\alpha + 1)$, we will choose the $l_i$'s as quadratic residues
in $\mu_{q-1} = \F_{p^{2\alpha + 1}}^*$. The following proposition justifies this choice.

\begin{prop}
Let $q = p^{2\beta +1}$, and define $\mu_{q-1} = \F_q^*$, $\mu_{p-1} = \F_p^*$. Then, any set of representatives of
$\frac{\mu_{q-1}^2}{\mu_{p-1}^2}$ is also a set of representatives of
$\frac{\mu_{q-1}}{\mu_{p-1}}$.
\end{prop}

\begin{proof}
Clearly, the number representatives in a slice of $\frac{\mu_{q-1}^2}{\mu_{p-1}^2}$ 
is the same as the number of representatives in a slice
of $\frac{\mu_{q-1}}{\mu_{p-1}}$.
It suffices to show that: if $x,y \in \mu_{q-1}^2$ are inequivalent mod $\mu_{p-1}^2$,
then they are inequivalent mod $\mu_{p-1}$.\footnote{In the
multiplicative sense} 
Thus, suppose otherwise; i.e., $x,y$ are inequivalent mod
$\mu_{p-1}^2$ but they are equivalent mod $\mu_{p-1}$. Let
$x = y \psi$, where $\psi \in \mu_{p-1}\backslash \mu_{p-1}^2$.
Clearly, the quadratic residue symbol $\QRSym{\cdot}$ is a 
multiplicative
character of $\mu_{q-1}$. Hence,
\begin{eqnarray}
\nonumber
\begin{array}{ccccccc}
  1 &=& \QRSym{x} &=& \QRSym{y} \QRSym{\psi} &=& \QRSym{\psi}.
\end{array}
\end{eqnarray}

Therefore, $\psi$ is a quadratic residue in $\mu_{q-1}$. Thus,
$\mu_{p-1} \subset \mu_{q-1}^2$. The following claim will
give us a contradiction.

\begin{claim}
Let $p$ be a prime and $q = p^{\beta}$, then:
 $\F_p^* \subset (\F_q^*)^2$ if and only if $q = p^{2\alpha}$
is an even power of $p$
\end{claim}

\begin{proof}
($\Rightarrow$) 
Let $\eta \in (\F_p^*) \backslash (\F_p^*)^2$ and choose $\psi \in (\F_q^*)^2$ so that $\psi^2 = \eta$, then the splitting
field of $\psi$ over $\F_p$ must be contained in $\F_q$; Hence,
$\Z/ 2\Z$ must be a subgroup of $Gal(\F_q\mid \F_p)$, the Galois Group of
$\F_q$ over $\F_p$. Note that 
$Gal(\F_q,\F_p) = (\Z/ \beta\Z)$; thus, $2$ must divide $\beta$ since
$(\Z/ 2\Z) \subset (\Z/ \beta\Z)$.

($\Leftarrow$)
Let $\eta \in (\F_p^*) \backslash (\F_p^*)^2$ and choose $\psi$ so that $\psi^2 = \eta$, then the splitting
field $\F$ of $\psi$ over $\F_p$ has degree $2$; hence, it has Galois
Group $(\Z/ 2\Z)$ and $\F$ has order $p^2$. Note that the Galois Group 
of $\F_q$ over
$\F_p$ is $(\Z/ 2\alpha\Z)$. Hence, by Galois theory, $\F_q$ has a subfield $\F'$ 
with Galois Group $(\Z/ 2\Z)$ and order $p^2$. Since the field 
of $p^2$ elements
is unique modulo field isomorphisms, it must follow that $\F'$ is
the splitting field of $\psi$. Thus, $\psi \in (\F_q^*)^2$ and
the conclusion follows.
\end{proof} 

By the previous claim, we arrive at a contradiction.
\end{proof}

We will identify the group $H = (\Z/ p^2\Z)^{\beta}$ with the Galois Ring $GR(p^2,\beta)$, and we will use the special pairing $\theta$ induced by the trace function
of the $GR(p^2,\beta)$. That is,
\begin{eqnarray}
 \nonumber
 \theta(g')(g) &=& \eta_{p^2}^{Tr(g'*g)},
\end{eqnarray}
where we have identified $H$ with $GR(p^2,\beta)$, the operation $*$ is multiplication in the Galois Ring $GR(p^2,\beta)$, and $Tr(\cdot)$ is the Galois trace in $GR(p^2,\beta)$.
We note that the choice of $Tr(\cdot)$ depends on the embedding of
$H = (\Z/ p^2\Z)^{\beta}$ in the Galois Ring $GR(p^2,\beta)$ and
the choice of $\eta_{p^2}$. 
We will denote $Tr(g'*g)$ by $\LRGen{g',g}_H$. Clearly, using the pairing $\theta$ induced
by $Tr(\cdot)$, 
\begin{eqnarray}
\label{newagg1}
 A_{H,H_1}(\Omega_{\theta(g')},\Omega_{g}) &=& \left\{
\begin{array}{cc}
 \QRSym{n} o(p \cdot g) & \text{ if $\LRGen{g',g}_H = n p$ mod $p^2$ for some $n$}\\
  & \text{  where $(n,p) = 1$}\\
 0 & \text{ else}
\end{array}
\right.
\end{eqnarray}
Also, by restricting $Tr(\cdot)$ from $H$ to $L = p \cdot H$, we can induce a pairing on $\theta'$ on $L$ and define $A_{L,L_1}$. Clearly,
$Tr\mid_{L}$ is the regular trace function of the finite field $\F_{p^\beta}$. We will denote $Tr\mid_{L}(g'*g)$ by $\LRGen{g',g}_{L}$.

Now, we proceed to calculate $A_{H,H_1}$ for $H = (\Z/ p^2)^{\beta}$. 

\begin{prop}
\label{ahh1form}
Let $H = (\Z/p^2 \Z)^{\beta}$ and $L = p\cdot H = (\Z/p \Z)^{\beta}$. Then,
there is an ordering of the orbits of the action of $H_1$ on 
$H\backslash \{ 0 \}$, such that:
\begin{eqnarray}
\nonumber
A_{H,H_1} &=&
 	 \left(
		\begin{array}{cccc}
		0 & p A_{L,L_1} & \cdots & p A_{L,L_1}\\
		A_{L,L_1} & p J_{H,1,1} & \cdots & pJ_{H,1,p^{\beta -1}}\\
		\vdots & \vdots & \cdots & \vdots \\
		A_{L,L_1} & p J_{H,p^{\beta -1},1} & \cdots 
			& pJ_{H,p^{\beta-1},p^{\beta -1}}\\
		\end{array}
	\right)
	\begin{array}{c c}
		\} & m_{\beta,p} \\
		\Biggr{\}} & 
		  p^{\beta -1}m_{\beta,p}
	\end{array},
\end{eqnarray}
where $J_{H,i,j}$ has nonzero support on the zero pattern 
of $A_{L,L_1}$,
and $m_{\beta,p} = p^{\beta -1} + p^{\beta -1} + \cdots + p + 1$.
\end{prop}

\begin{proof}
We will organize the $H_1$ orbit representatives according to the
parametrization of Proposition \ref{hzp2a}. 
Using the notation of Proposition \ref{hzp2a}, consider the
following enumeration of the orbits:

\begin{displaymath}
\begin{array}{lr}
\nonumber
 \{ p \cdot l_1, \ldots, p \cdot l_{r'} \} &  \text{ for elements of order } p\\
\nonumber
 \{ l_1 + p\cdot l_{1,1}',\ldots , l_{r'} + p \cdot l_{r',1}, l_1 + p\cdot l_{1,2}',\ldots , 
  &  \text{ for elements of order } p^2 \\
\nonumber
  l_{r'} + p \cdot l_{r',2}
  ,\ldots , l_1 + p\cdot l_{1,p^{\beta -1}}',\ldots , l_{r'} +
	 p \cdot l_{r',p^{\beta -1}} \} & 
\end{array}.
\end{displaymath}
Where $r' = m_{\beta,p}$. Using the above enumeration, we can calculate $A_{H,H_1}$ as the following:
\begin{eqnarray}
\nonumber
A_{H,H_1} &=& \left(
		\begin{array}{cccc}
		A_{0,0} & p A_{0,1} & \cdots & p A_{0,p^{\beta -1}}\\
		A_{1,0} & p A_{1,1} & \cdots & p A_{1,p^{\beta -1}}\\
		\vdots & \vdots & \cdots & \vdots \\
		A_{p^{\beta-1},0} & p A_{p^{\beta-1},1} & \cdots & 
			p A_{p^{\beta-1},p^{\beta -1}}\\
		\end{array}
	\right),
\end{eqnarray}
where: the $A_{i,j}$'s are all $r' \times r'$ matrices; 
$A_{0,j}$, where $1 \leq j$, is defined on tuples of the form 
$(p\cdot l_k, l_t + p \cdot l_{t,j}')$;
$A_{i,0}$, where $1 \leq i$, is defined on tuples of the form
$(l_k + p \cdot l_{k,i}', l_t)$;
$A_{0,0}$ is defined on tuples of the form 
$(p\cdot l_k, p \cdot l_t)$;
and $A_{i,j}$, where $1 \leq i, 1\leq j$, is defined on tuples of the form
$(l_k + p \cdot l_{k,i}', l_t + p \cdot l_{t,j}')$.

We proceed to calculate the $A_{i,j}$'s using Equation (\ref{newagg1}).
First, we calculate $A_{0,0}$. Since $A_{0,0}$ is valued on tuples of the form
$(p \cdot l_k, p\cdot l_t)$ and
\begin{eqnarray} 
\nonumber
\LRGen{p\cdot l_k, p \cdot l_t}_H &=& p^2 \LRGen{l_k,l_t}_H\\
\nonumber
	&=& 0 \text{ mod } p^2
\end{eqnarray} 
By Equation (\ref{newagg1}), $A_{0,0} = 0$.

Now, we calculate $A_{i,0}$ ($1 \leq i$). Since $A_{i,0}$ is
valued on tuples of the form $(l_k + p \cdot l_{k,i}', p \cdot l_t)$
and
\begin{eqnarray} 
\nonumber
\LRGen{l_k + p \cdot l_{k,i}', p \cdot l_t}_H &=& 
 p\LRGen{l_k,l_t}_H + p^2\LRGen{l_{k,i}',l_t}_H\\
\nonumber
  &=& p\LRGen{l_k,l_t}_H \text{ mod } p^2 \\
\nonumber
  &=& p\LRGen{\overline{l_k},\overline{l_t}}_L \text{ mod } p^2,
\end{eqnarray}
where $\overline{l_k}$ is the result of taking all the coordinates of $l_k$ mod $p$.
Hence, by using Equation (\ref{newagg1}), we have $A_{i,0} = A_{L,L_1}$.
The proof that $A_{0,j} = A_{L,L_1}$ is similar to this case.

It suffices to show that $A_{i,j}$, where $1 \leq i, 1 \leq j$, has
 nonzero support on
the zero pattern of $A_{L,L_1}$. Clearly, $A_{i,j}$ is valued on the
tuples of the form $(l_k + p \cdot l_{k,i}', l_t + p \cdot l_{t,j}')$ and
\begin{eqnarray}
\nonumber
\LRGen{l_k + p \cdot l_{k,i}', l_t + p \cdot l_{t,j}'}_H &=& 
	\LRGen{l_k,l_t}_H + p(\LRGen{l_{k,i}',l_t}_H + 
		\LRGen{l_t,l_{t,j}'}_H)\\
\nonumber
	& &	+ p^2 \LRGen{l_{k,i}',l_{t,j}'}_H \\
\nonumber
	&=& \LRGen{l_k,l_t}_H \\
\nonumber
	& &	+ p(\LRGen{l_{k,i}',l_t}_H + \LRGen{l_t,l_{t,j}'}_H)
		\text{ mod } p^2. 
\end{eqnarray}
Hence, if $A_{i,j}(l_k + p \cdot l_{k,i}', l_t + p \cdot l_{t,j}') \neq 0$, then
$ \LRGen{l_k + p \cdot l_{k,i}', l_t + p \cdot l_{t,j}'}_H = p f$ mod $p^2$ where 
$f \in \{1,\ldots, p-1\}$. Thus,
\begin{eqnarray}
\nonumber
pf &=& \LRGen{l_k + p \cdot l_{k,i}', l_t + p \cdot l_{t,j}'}_H \\
\nonumber
	&=& \LRGen{l_k,l_t}_H + p(\LRGen{l_{k,i}',l_t}_H + \LRGen{l_t,l_{t,j}'}_H)
		\text{ mod } p^2 ,
\end{eqnarray}
hence, $\LRGen{l_k,l_t}_H = 0$ mod $p$. But $\LRGen{l_k,l_t}_H = 
\LRGen{\overline{l_k},\overline{l_t}}_L$
mod $p$. Thus, if 
$A_{i,j}(l_k + p \cdot l_{k,i}', l_t + p \cdot l_{t,j}') \neq 0$,
then $\LRGen{\overline{l_k},\overline{l_t}}_L = 0$ mod $p$. That is, 
$A_{i,j}$ is nonzero on a subset of the zero pattern of $A_{L,L_1}$.
\end{proof}

\begin{cor}
\label{ahh1form2}
Let $H$ and $A_{H,H_1}$ be given as they are in Proposition \ref{ahh1form}.
Let
\begin{eqnarray}
\nonumber
 B_H' &=& \left(
		\begin{array}{ccc}
			J_{H,1,1} & \cdots & J_{H,1,p^{\beta -1}}\\
			\vdots & \cdots & \vdots\\
			J_{H,p^{\beta-1},1} & \cdots 
				& J_{H,p^{\beta-1},p^{\beta -1}}\\
		\end{array}
	\right),
\end{eqnarray}
and $B_H = p B_H'$. Then,
 \begin{enumerate}
   \item The following holds $B_H' J = 0$, i.e., zero row sums; and, 
	$\sum_{j=1}^{p^{\beta-1}} J_{H,i,j} = 0$.
   \item The following holds $J B_H' = 0$, i.e., zero column sums; and, 
	$\sum_{i=1}^{p^{\beta-1}} J_{H,i,j} = 0$.
   \item The following holds $B_H^3 = \frac{|H|}{p} B_H = p^{2\beta -1} B_H$.
   \item The following holds $B_H^2 = p^{2\beta -1} I -
	p^{\beta} J_{p^{\beta-1},p^{\beta-1}} \otimes I_{r'}$.
 \end{enumerate}
\end{cor}

\begin{proof}
This is a consequence of $A_{H,H_1}^2 = p^{2\beta -1} I$.
Consider,
\begin{eqnarray}
\label{ahh1iden1} 
 p^{2\beta -1} I &=&
\end{eqnarray}
\begin{displaymath}
\begin{array}{l}
\nonumber
	\left(
		\begin{array}{cccc}
		0 & p A_{L,L_1} & \cdots & p A_{L,L_1}\\
		A_{L,L_1} & p J_{H,1,1} & \cdots & pJ_{H,1,p^{\beta -1}}\\
		\vdots & \vdots & \cdots & \vdots \\
		A_{L,L_1} & p J_{H,p^{\beta -1},1} & \cdots 
			& pJ_{H,p^{\beta-1},p^{\beta -1}}\\
		\end{array}
	\right)
	\left(
		\begin{array}{cccc}
		0 & p A_{L,L_1} & \cdots & p A_{L,L_1}\\
		A_{L,L_1} & p J_{H,1,1} & \cdots & pJ_{H,1,p^{\beta -1}}\\
		\vdots & \vdots & \cdots & \vdots \\
		A_{L,L_1} & p J_{H,p^{\beta -1},1} & \cdots 
			& pJ_{H,p^{\beta-1},p^{\beta -1}}\\
		\end{array}
	\right), 
\end{array}
\end{displaymath}
Equation (\ref{ahh1iden1}) gives:
\begin{eqnarray}
\nonumber
 p(\sum_{j=1}^{p^{\beta-1}} J_{H,i,j})A_{L,L_1} &=& 0,\\ 
\nonumber
 pA_{L,L_1}(\sum_{i=1}^{p^{\beta-1}} J_{H,i,j}) &=& 0.
\end{eqnarray}
From which part 1 and part 2 follow.

Note that Equation (\ref{ahh1iden1}) also gives the following equation.
\begin{eqnarray}
\nonumber
 pA_{L,L_1}^2 + \sum_{k=1}^{p^{\beta-1}} (pJ_{H,i,k})(pJ_{H,k,j}) &=&
	\delta_{i,j}p^{2\beta -1} I.
\end{eqnarray}
Since $A_{L,L_1}^2 = p^{\beta -1} I$, we get that
\begin{eqnarray}
\label{ahh1iden2}
 \sum_{k=1}^{p^{\beta-1}} (pJ_{H,i,k})(pJ_{H,k,j}) &=&
	\delta_{i,j}p^{2\beta -1} I - p^{\beta} I.
\end{eqnarray}
From which part 3 follows.

To show part 4, apply $B_H$ to Equation (\ref{ahh1iden2}),
\begin{eqnarray}
\nonumber
 \sum_{k=1,j=1}^{p^{\beta-1}} (pJ_{H,i,k})(pJ_{H,k,j})(pJ_{H,j,l}) &=&
	\sum_{j=1}^{p^{\beta-1}}(\delta_{i,j}p^{2\beta -1} I - p^{\beta} I)
		(pJ_{H,j,l})\\
\nonumber
 &=& \sum_{j=1}^{p^{\beta -1}}\delta_{i,j}p^{2\beta -1}(pJ_{H,j,l})
 - p^{\beta}  \sum_{j=1}^{p^{\beta-1}} (pJ_{H,j,l}) \\
\nonumber
 &=& \sum_{j=1}^{p^{\beta -1}}\delta_{i,j}p^{2\beta -1}(pJ_{H,j,l})\\
\nonumber
 &=& p^{2\beta -1}(pJ_{H,i,l}),
\end{eqnarray}
from which part 4 follows.
\end{proof}

Now, we proceed to construct the element $L_0(x)$ of Proposition (\ref{gshds3_prop1}). We will consider the groups
$H = (\Z/ p^2\Z)^{2\alpha +1 }$ and choose the $l_i$'s in Proposition (\ref{hzp2a}) as the quadratic
residues in $\F_{p^{2\alpha +1}}$. We note that the set of quadratic residues in $(\F_{p^{2\alpha + 1}}, +)$ is a GSHDS. We will omit
the proof of this statement. This choice for the $l_i$'s will impose conditions on the elements of the matrix $A_{H,H_1}$ as it is shown
in the next proposition.

\begin{prop}
\label{SpcCaseP4}
Let $\underline{j}$  be the vector 
 all $1$'s.
 Using the assumptions on the $l_i$'s, the results and the notation of 
 Propositions (\ref{ahh1form}) and Corollary (\ref{ahh1form2}); where
 $H= (\Z/ p\Z)^{2 \alpha+1}$ and 
 $L= p \cdot H \simeq (\Z/ p\Z)^{2\alpha+1}$, the following hold: 
 \begin{enumerate}
  \item We have that $J_{H,s,t} \underline{j} = \lambda_{s,t} \underline{j}$
 for some integer $\lambda_{s,t}$. More percisely,
 \begin{eqnarray}
  \nonumber 
    \lambda_{s,t} &=& \sum_{j=1}^{m_{2\alpha +1,p}}\QRSym{Tr(l_j(1+ p(l_s' + l_t')))},
 \end{eqnarray}
  where we have extended $\QRSym{\cdot}$ so that $\QRSym{n} = 0$ whenever $p$ divides $n$, and
$m_{2\alpha +1,p} = \frac{p^{2\alpha+1}-1}{p-1}$.
  \item We have that $A_{L,L_1} \underline{j} = \epsilon_0 p^{\alpha} \underline{j}$
where $\epsilon_0 \in \{ 1, -1 \}$.
  \item Let $L = [[ \lambda_{s,t} ]]$, the matrix formed by the 
$\lambda_{s,t}$'s. Then,
  \begin{eqnarray}
\nonumber
 L^2 &=& p^{2\alpha -1} ( p^{2\alpha} I - J).
  \end{eqnarray}
  \item We have that $L^T = L$.
 \end{enumerate}
\end{prop}

\begin{proof}
Part 1 is a direct consequence of our choice of $l_i$'s and the choice of $\theta(g)(g')$. Clearly, the $l_i$'s form a group
under Galois Ring multiplication.
Also, $J_{H,s,t}$ is a matrix on the tuples
$(l_i(1+pl_s'),l_j(1+pl_t'))$; hence, the $i$th coordinate
of $J_{H,s,t}\underline{j}$ is equal to:
\begin{eqnarray}
 \nonumber
 \sum_{j=1}^{m_{2\alpha +1,p}} \left(\frac{Tr(l_i(1+pl_s')l_j(1+pl_t'))}{p}
	\right) &=&
 \sum_{j=1}^{m_{2\alpha+1,p}} \left(\frac{Tr(l_il_j(1+p(l_s'+l_t')))}{p}
	\right).
\end{eqnarray}
Note that 
$(1+p(l_s' + l_t'))$ is fixed in the above sum. Also,
$l_i l_j = l_{\sigma(j)} \alpha_{i,j}$ where: $\sigma$ is the induced
permutation action of $l_i$ on $\frac{\mu_{q-1}^2}{\mu_{p-1}^2}$, and
$\alpha_{i,j} \in \mu_{p-1}^2$ is some quadratic residue mod $p$.
Clearly,
\begin{eqnarray}
  \nonumber
  \QRSym{Tr(l_i l_j (1+ p(l_s' + l_t') ))} &=& \QRSym{Tr(l_{\sigma(j)} \alpha_{i,j}(1+p(l_s' + l_t') ))}\\
 \nonumber
  &=& \QRSym{\alpha_{i,j}}\QRSym{Tr(l_{\sigma(j)}(1+p(l_s' + l_t')))}\\
  \nonumber
  &=& \QRSym{Tr(l_{\sigma(j)}(1+p(l_s' + l_t')))}\\
  \nonumber
  &=& \QRSym{Tr(l_{\sigma(j)}l_{s,t}'')},
\end{eqnarray}
where $l_{s,t}'' = 1 + p (l_s' + l_t')$. Hence, the $i$th coordinate
of $J_{H,s,t}\underline{j}$ is equal to:
\begin{eqnarray}
 \nonumber
 \sum_{j=1}^{m_{2\alpha+1,p}} \QRSym{Tr(l_i l_j (1 + p(l_s' + l_t')))} &=&
  \sum_{j=1}^{m_{2\alpha+1,p}} \QRSym{Tr(l_{\sigma(j)}l_{s,t}'')}\\
 \nonumber
  &=& \sum_{j=1}^{m_{2\alpha+1,p}} \QRSym{Tr(l_{j}l_{s,t}'')},
\end{eqnarray}
which is a value that is independent of $i$ but dependent on $s,t$. 
This shows part 1.

Now, we proceed to show part 2. Note that by the choice of 
the $l_i$'s, the vector of all $1$'s 
,$\underline{j}$, represents the Quadratic Residues 
$(\F_{p^{2\alpha+1}}^*)^2$ 
of $L \simeq (\F_{p^{2\alpha +1}},+)$. 
Hence, by Proposition (\ref{prop3}) part 3,
$A_{L,L_1} \underline{j} = p^{\alpha} \overline{d}$ for some vector $\overline{d}$ 
of $\pm 1$s, since $\underline{j}$ is a GSHDS in $L$. Clearly, the $i$th coordinate of $A_{L,L_1} \underline{j}$ is given by:
\begin{eqnarray}
 \nonumber
 \sum_{j=1}^{m_{2\alpha +1,p}} \QRSym{Tr(\overline{l_i} \overline{l_j})},
\end{eqnarray}
where $\overline{l_i} = l_i$ mod $p$, and 
$\overline{l_j} = l_j$ mod $p$. Clearly, $\{ \overline{l_1},\cdots
,\overline{l_{m_{2\alpha +1,p}}} \} = \frac{(\F_{p^{2\alpha+1}}^*)^2}{(\F_p^*)^2}$
is a group; and, by using a similar reasoning as part 1, we can deduce
that the $i$th coordinate of $A_{L,L_1}\underline{j}$ is given
by:
\begin{eqnarray}
 \nonumber
 \sum_{j=1}^{m_{2\alpha +1,p}} \QRSym{Tr(\overline{l_i} \overline{l_j})} &=&
 \sum_{j=1}^{m_{2\alpha +1,p}} \QRSym{Tr(\overline{l_{\sigma(j)}})} \\
 \nonumber
 &=& \sum_{j=1}^{m_{2\alpha +1,p}} \QRSym{Tr(\overline{l_{j}})}\\
 \nonumber
 &=& \epsilon_0 p^{\alpha},
\end{eqnarray}
where $\epsilon_0 = \pm 1$. Thus, the $i$th
coordinate $A_{L,L_1}\underline{j}$ is independent
of $i$ and the second part follows.

To show part 3, consider the following equation from 
Corollary \ref{ahh1form2}.
\begin{eqnarray}
\nonumber
 \sum_{k=1}^{m_{2\alpha+1,p}} (pJ_{H,i,k})(pJ_{H,k,j}) &=&
	\delta_{i,j}p^{4\alpha+1} I - p^{2\alpha + 1} I
\end{eqnarray}
We can deduce part 3 by multiplying the previous equation by $\underline{j}$ and simplifying.
Part 4 follows from $J_{H,s,t} = J_{H,t,s}$. 
\end{proof}

We will identify each block $J_{H,i,j}$ of $A_{H,H_1}$ with
the element $l_j' \in \frac{ \mu_{q-1} \cup \{ 0 \} }{ \mu_{p-1}
\cup \{ 0 \}} \simeq 
\frac{(\F_q,+)}{( \F_p,+)} \simeq (\Z/ p\Z)^{2\alpha}$.
Under this identification,
the matrix $L$ of Proposition \ref{SpcCaseP4} can be thought
as a matrix over entries indexed by the elements
of  $(\F_{p^{2\alpha + 1}}, + )$ mod
$(\F_p, +)$. That is, as a matrix indexed by
elements in $(\F_{p^{2\alpha}},+) \simeq
(\Z/ p\Z)^{2\alpha} = K$. Thus, for each $l_j'$ denote $k_j$ the
corresponding element in $K$. We will assume an
ordering of the $l_j'$'s such that $l_1'$ corresponds to
the $0$ element of $K$.

Let us introduce the following element of $\Z[K]$.
\begin{eqnarray}
\nonumber
 L_0(x) &=& \sum_{k_i \in K} \lambda_{1,i} x^{k_i}
\end{eqnarray}

Where we have assumed that $k_1 = 0$ the additive identity of $K$.
We will show some properties of $L_0(x)$, and we will use
$L_0(x)$ to establish necessary existence conditions for GSHDS in $G$.

\begin{prop}
 Let $\rho_K$ be the regular representation of $K$ with
action on $Z = \bigoplus_{i=1}^{p^{2\alpha}} \C e_{k_i}$. Then:
 \begin{enumerate}
  \item For $g \in K$, $\rho_K(g^{-1}) L = L \rho_K(g)$, as matrices.
  \item For $\chi \in \overline{K}$, 
   $L e_{\chi} = \overline{\chi}(L_0) e_{\overline{\chi}}$ where:
	\begin{eqnarray}
	\nonumber
		e_{\chi} &=& \frac{1}{p^{2\alpha}} 
		\sum_{k \in K}  \overline{\chi}(k) e_k.
	\end{eqnarray}
  \item For $\chi \in \overline{K}$, $\chi$ non principal. 
   $\chi(L_0) \overline{\chi}(L_0) = p^{4\alpha -1 }$. 
  \item Let $\chi_0$ be the principal character of $K$, then
  $\chi_0 (L_0) = 0$.
  \item $L_0(x) L_0(x^{-1}) = p^{4\alpha -1} [1] - p^{2\alpha -1} K(x)$,
  in the group algebra of $K$.
 \end{enumerate}
\end{prop}

\begin{proof}
We proceed to show 
part 1. Let $k_j \in K$ and $l_j'$ the corresponding
element to $k_j$. Consider $\sigma = 1+ p l_j' \in H$, clearly $\sigma$ can be
viewed as an element of $Aut(H)$ by $\sigma(g) = (1+p l_j')*g$. Note that
$(\sigma)^* = \sigma$, because,
\begin{eqnarray}
 \nonumber
 Tr(g \sigma(g')) &=& Tr(g(1+p l_j') g')\\
 \nonumber
	&=& Tr( (1+pl_j') g g')\\
 \nonumber
  &=& Tr(\sigma(g) g').
\end{eqnarray}
Hence, $((\sigma)^*)^{-1} = (\sigma)^{-1} = 1- pl_j'$.

By Proposition \ref{prop2_1}, it follows that:
\begin{eqnarray}
 \label{SpcCaseE1}
  \rho_X(1- pl_j') A_{H,H_1} &=& A_{H,H_1} \rho_X(1+pl_j')
\end{eqnarray}

We note that $\rho_X(1 + pl_j')$ permutes the blocks
$J_{H,s,t}$, and it acts like $\rho_K(k_j)$ when we identify 
the blocks $J_{H,s,t}$ with their corresponding elements
of $K$. Similarly, $\rho_X(1-pl_j')$ acts like
$\rho_K(k_j^{-1})$ on the blocks $J_{H,s,t}$.

We will show 
\begin{eqnarray}
\label{SpcCaseE2}
\rho_K(k_j^{-1}) L &=& L \rho_K(k_j),
\end{eqnarray} 
by showing
that the columns of the left hand side equal to the columns
of the right hand side.
Thus, consider the column of the left
hand side of Equation (\ref{SpcCaseE2})
that corresponds to $k_s$. Clearly, this equals to
\begin{eqnarray}
 \nonumber
 \rho_K(k_j^{-1}) \LRVec{\begin{array}{c} \lambda_{1,s} \\ \vdots
\\ \lambda_{p^{2\alpha},s} \end{array}}.
\end{eqnarray}
Clearly, this is equal to the left hand side of 
Equation (\ref{SpcCaseE1}) after we multiply it on the right
with the vector:
\begin{eqnarray}
 \nonumber
 e_{k_s} &=& \LRVec{ \begin{array}{c}
	0\\
	\vdots\\
        0\\
	1\\
	1\\
	\vdots \\
	1\\
	0\\
	\vdots\\
	0	
	\end{array} }, 
\end{eqnarray}
where the $1$'s corresponds to the block $\{l_1(1+pl_s'),
\ldots, l_{r'}(1+pl_s') \}$. 

Note that when we multiply the right hand side of
Equation (\ref{SpcCaseE1})  by $e_{k_s}$ on the right, we get
the column that corresponds to $k_s$ in the right hand side
of Equation (\ref{SpcCaseE2}). This establishes part 1.

For part 2, consider the matrices:
\begin{eqnarray}
 \nonumber
 P_{\chi} &=& \frac{1}{p^{2\alpha}} \sum_{k \in K} \overline{\chi}(k)
		\rho_K(k),
\end{eqnarray}
where $\chi \in \overline{K}$. Clearly, the set of $P_{\chi}$ forms a complete
set of orthogonal projections. That is, 
$\sum_{\chi \in \overline{K}} P_{\chi} = I$ and
$P_{\chi} P_{\chi'} = \delta_{\chi,\chi'} P_{\chi'}$.
Also, one can calculate $P_{\chi} e_{\chi'} = 
\delta_{\chi,\chi'} e_{\chi'}$.

By part 1, we can show $P_{\overline{\chi'}} L = L P_{\chi'}$ for
all $\chi'$.
Hence,
\begin{eqnarray}
 \nonumber
  P_{\chi'} L e_{\chi} &=& L P_{\overline{\chi}'} e_{\chi}\\
 \nonumber
   &=& \left\{ \begin{array}{cc}
		Le_{\chi} & \text{ If } \chi' = \overline{\chi},\\
		0 & \text{ else}.
	\end{array} \right.
\end{eqnarray}
Thus, $L e_{\chi}$ belongs to $Im(P_{\overline{\chi}}) = 
\LRGen{e_{\overline{\chi}}}$. Hence,
\begin{eqnarray}
 \nonumber
  L e_{\chi} &=& \alpha e_{\overline{\chi}}
\end{eqnarray}
for some complex number $\alpha \in \C$. We can calculate the number $\alpha$
by considering the first row of $L e_{\chi}$. Clearly, this row
has value:
\begin{eqnarray}
\nonumber
L_0^T e_{\chi} &=& \LRGen{L_0,e_{\overline{\chi}}}\\
\nonumber
 &=& \frac{\overline{\chi}(L_0)}{p^{2\alpha}}.
\end{eqnarray}
On the other hand the first row of $\alpha e_{\overline{\chi}}$ has
value $\frac{\alpha}{p^{2\alpha}}$.
Thus, $\alpha = \overline{\chi}(L_0)$. Hence, part 2 follows.

Now, we proceed to show part 3. Note that by part 2:
\begin{eqnarray}
 \nonumber
 L e_{\chi} &=& \overline{\chi}(L_0) e_{\overline{\chi}},\\
 \nonumber
 L e_{\overline{\chi}} &=& \chi(L_0) e_{\chi},
\end{eqnarray}
hence,
\begin{eqnarray}
 \nonumber
 \chi(L_0) \overline{\chi}(L_0) &=& \chi(L_0) \overline{\chi}(L_0) 
	e_{\chi}^T e_{\overline{\chi}}\\
 \nonumber
 &=& (L e_{\overline{\chi}})^T (L e_{\chi}) \\
 \nonumber
 &=& e_{\overline{\chi}}^T L^T L e_{\chi} \\
 \nonumber
 &=& e_{\overline{\chi}}^T L^2 e_{\chi} \\
 \nonumber
 &=& e_{\overline{\chi}}^T ( p^{2\alpha -1}(p^{2\alpha} I - J ) ) e_{\chi} \\
 \nonumber
 &=& p^{4\alpha-1} e_{\overline{\chi}}^T e_{\chi}\\
 \nonumber
 &=& p^{4\alpha-1}.
\end{eqnarray}
Thus, part 3 follows.

Clearly, part 4 follows from Corollary (\ref{ahh1form2}), since
the row sums of $B_H'$ are zero. This forces the row
sums of $L$ to be zero. In particular, the sum of all the entries
of $L_0$ are zero.
Part 5 follows from parts 3 and 4 by using Fourier Inversion.
\end{proof}

We note that the calculation of $L_0(x)$ depends on the choice of the embedding of $(\Z/ p^2\Z)^{2\alpha +1}$ in $GR(p^2,2\alpha +1)$,
the choice of the quadratic residues $l_i$'s, and the choice of the trace function. We leave the direct
calculation of $L_0$ as an open problem.

\section{Necessary Existence Conditions for $(\Z/ p^2\Z)^{2\alpha +1 } \times (\Z/ p\Z)$}

The previous section derived the calculation of the element $L_0$. In this section, we will use $L_0$ to give necessary existence conditions for the family of
groups $(\Z/ p^2\Z)^{2\alpha +1} \times (\Z/ p\Z)$.

\begin{prop}
\label{GSHDS_SpcCsP1}
Let $D$ be a GSHDS in $(\Z/p \Z) \times ( \Z/ p^2 \Z)^{2\alpha +1}$, $H = \{0\} \times (\Z/ p^2\Z)^{2\alpha +1} \subset G$, 
$L = p\cdot H = (\Z/ p\Z)^{2\alpha +1}$, and $K = \frac{L}{(\Z/ p\Z)}$. Then, 
there are elements $A$ and $B$ in $\Z[K]$ such that:
\begin{enumerate}
 \item The following holds, $p^{2\alpha} A(x) = \chi_0 (A) K(x) + L_0(x) * B(x^{-1})$, where $\chi_0$
 is the principal character of $K$.
 \item The following holds, $p^{2\alpha} B(x) = \chi_0 (B) K(x) + p L_0(x) * A(x^{-1})$, where $\chi_0$
 is the principal character of $K$.
 \item The following holds, $\chi_0(A) = p^{\alpha -1} \epsilon_0 b_0$, where $\epsilon_0$ is 
  $\pm 1$ and $b_0$ is an integer.
 \item The following holds, $\chi_0(B) = p^{\alpha} \epsilon_0 a_0$; where $\epsilon_0$ is 
  $\pm 1$, and equal to the $\epsilon_0$ of part 3; and, $a_0$ is an integer.
 \item All the coefficients of $A(x)$ are odd and $\chi_0(A)$ is odd.
 \item All the coefficients of $B(x)$ are odd and $\chi_0(B)$ is odd.
 \item The numbers $a_0$ and $b_0$ are odd integers.
\end{enumerate}
\end{prop}

\begin{proof}
 Let us assume that there is a GSHDS $D$ in 
$G =(\Z/p \Z) \times (\Z/ p^2 \Z)^{2\alpha +1}$. 
Let $D' = H \cap D$ and  $D'' = L\cap D$.
Let $d'$ be the vector of $\pm 1$s representing $D'$ 
using the $H_1$ orbit representatives of $\overline{H}$ in 
Proposition \ref{hzp2a}.
Clearly, $d' = (d_0'^T,d_1'^T,\ldots, d_{p^{2\alpha}}'^T)$,
where $d_i'$ is a 
$(p^{2\alpha} + \cdots + p +1) \times 1$ vector of $\pm 1$s. 

Consider the canonical projection $\pi : G = (\Z/ p\Z) \times H \rightarrow H$. Clearly,
this projection has kernel $M = (\Z/ p\Z)$. Apply Proposition (\ref{prop8}) to  $D$ and $\pi$. We have that,
\begin{eqnarray}
\nonumber
 A_{H,H_1} \nu_{G,M}(D) &=& df_G(D,H)
\end{eqnarray}
Where $\nu_{G,M}(D)$ are the difference intersection numbers of $D$ with respect to $M$, and $df_G(D,H)$ are the difference
coefficients of $D$ with respect to the characters of $H$ extended to $G$ by $\chi \rightarrow \chi \circ \pi$. Clearly, by 
Proposition (\ref{prop3}) part 3, $df_G(D,H) = p^{\alpha} d_{\overline{D} \cap \overline{H}}$ where $d_{\overline{D} \cap \overline{H}}$ is
the $\pm 1$s representation of the $\overline{D} \cap \overline{H}$, and $\overline{D}$ is the dual GSHDS of $D$ introduced in
Corollary (\ref{corDual}). Thus, we have,
\begin{eqnarray}
\nonumber
 A_{H,H_1} \nu_{G,M}(D) &=& p^{2\alpha +1} d_{\overline{D} \cap \overline{H}}
\end{eqnarray}
We will apply Proposition (\ref{prop3}) part 1 to the previous equation to deduce,
\begin{eqnarray}
\nonumber
 A_{H,H_1}d_{\overline{D} \cap \overline{H}}  &=&  p^{2\alpha } \nu_{G,M}(D)  
\end{eqnarray}
Clearly, $\nu_{G,M}(D)$ is a vector of odd numbers with absolute values $\leq |M| = p$. By Proposition (\ref{prop3}) part 2,
not all entries of $\nu_{G,M}(D)$ are divisible by $p$.
Applying the previous argument to the GSHDS $D = \overline{D}$, we deduce,
\begin{eqnarray}
\nonumber
 A_{H,H_1} d' &=&  A_{H,H_1}d_{D \cap H} \\
\nonumber
	 &=&  p^{2\alpha } \nu_{\overline{G},\overline{M}}(\overline{D})  \\
\nonumber
	&=& p^{2\alpha} \nu'
\end{eqnarray}
where $\nu' = (\nu_0'^T,\nu_1'^T,\ldots,\nu_{p^{2\alpha}}'^T)$ is a vector of
odd integers between $-p$ and $p$ whose entries are not all divisible by $p$.

Using the form of $A_{H,H_1}$ given
by Proposition \ref{ahh1form}, we get the following equations:
\begin{eqnarray}
\label{gr_cond2}
 p \sum_{j=1}^{p^{2\alpha}} A_{L,L_1}d_j' &=& p^{2\alpha}\nu_0',\\
\label{gr_cond3}
 A_{L,L_1} d_0' + p \sum_{j=1}^{p^{2\alpha}} J_{H,i,j}d_j' &=& p^{2\alpha}\nu_i'.
\end{eqnarray}
By taking the inner product of Equation (\ref{gr_cond2}) with $\underline{j}$ and
using the fact that $A_{L,L_1}^T = A_{L,L_1}$, we get
\begin{eqnarray}
\nonumber
 \sum_{j=1}^{p^{2\alpha}} p^{\alpha+1}\epsilon_0 
  \LRGen{d_j', \underline{j}} &=& p^{2\alpha} \LRGen{\nu_0',
 \underline{j}},
\end{eqnarray}
thus,
\begin{eqnarray}
\nonumber
 \sum_{j=1}^{p^{2\alpha}} 
    \LRGen{d_j', \underline{j}} &=& p^{\alpha -1}\epsilon_0 
		\LRGen{\nu_0',\underline{j}}.
\end{eqnarray}
Let $a = (a_1,\ldots,a_{p^{2\alpha}})^T$, where 
	$a_i = \LRGen{d_i', \underline{j}}$, and
let $b = (b_1,\ldots, b_{p^{2\alpha}})^T$, where 
	$b_i = \LRGen{\nu_i',\underline{j}}$.
Also, let 
  $a_0 = \LRGen{d_0', \underline{j}}$ and 
	$b_0 = \LRGen{\nu_0',\underline{j}}$.

We have shown that 
\begin{eqnarray}
\label{gr_cond_eqn1}
\LRGen{a,\underline{j}} &=& p^{\alpha -1}\epsilon_0 b_0.
\end{eqnarray}
Now, we will take the inner product of Equation (\ref{gr_cond3}) with
 $\underline{j}$, we get:
\begin{eqnarray}
\nonumber
 \LRGen{A_{L,L_1}d_0',\underline{j}} + 
	p \sum_{j=1}^{p^{2\alpha}} \LRGen{J_{H,i,j}d_j', 
  \underline{j}} &=& p^{2\alpha}\LRGen{\nu_i',\underline{j}},
\end{eqnarray}
where we used $J_{H,i,j}^T = J_{H,i,j}$. Hence, we have the following 
simplification:
\begin{eqnarray}
\nonumber
 \epsilon_0 p^{\alpha} \LRGen{d_0',\underline{j}} + 
 p \sum_{j=1}^{p^{2\alpha}} \lambda_{i,j} 
	\LRGen{d_j',\underline{j}} &=&
  p^{2\alpha} \LRGen{\nu_i',\underline{j}},
\end{eqnarray}
or,
\begin{eqnarray}
\nonumber
 \epsilon_0 p^{\alpha} a_0 +
 p \sum_{j=1}^{p^{2\alpha}} \lambda_{i,j} a_j &=&
  p^{2\alpha} b_i,
\end{eqnarray}
thus, we can deduce the following matrix equation:
\begin{eqnarray}
\label{gr_cond_eqn2}
 p^{2\alpha-1} b &=& p^{\alpha -1} \epsilon_0 a_0 \underline{j} +
	L b.
\end{eqnarray}

Performing the same analysis but starting with the equation 
$A_{H,H_1} \nu' = p^{2\alpha + 1} d'$, we get the following equations:
\begin{eqnarray}
\label{gr_cond_eqn3}
 \LRGen{b,\underline{j}} &=& p^{\alpha} \epsilon_0 a_0,\\
\label{gr_cond_eqn4}
 p^{2\alpha} a &=& \epsilon_0 p^{\alpha -1} b_0 \underline{j} + L b.
\end{eqnarray}

Let $A(x)$ be the element of $\Z[K]$ that has value $a_i$ at $k_i$, and
$B(x)$ be the element of $\Z[K]$ that has value $b_i$ at $k_i$.
We will use Equation (\ref{gr_cond_eqn2}) to establish a condition on
the character values of $A(x)$ and $B(x)$.

Let $\chi$ be a nonprincipal character of $K$. 
Equation (\ref{gr_cond_eqn2}) gives:
\begin{eqnarray}
\nonumber
 p^{2\alpha -1} \LRGen{b,e_{\chi}} &=&
  \epsilon_0 a_0 p^{\alpha -1} \LRGen{\underline{j},e_{\chi}} +
 \LRGen{La, e_{\chi}}.
\end{eqnarray}
Since $\chi$ is nonprincipal, we 
	have $\LRGen{\underline{j},e_{\chi}} = 0$. Thus,
\begin{eqnarray}
\nonumber
 p^{2\alpha -1} \LRGen{e_{\chi},b} &=& \LRGen{e_{\chi},La}\\
\nonumber
 &=& \LRGen{L e_{\chi},a}\\
\nonumber
 &=& \LRGen{\overline{\chi}(L_0) e_{\overline{\chi}},a} \\
\nonumber
 &=& \overline{\chi}(L_0) \LRGen{e_{\overline{\chi}},a}.
\end{eqnarray}
Since, for any $a$, 
$\LRGen{e_{\chi},a} = \frac{1}{p^{2\alpha}} \overline{\chi}(a)$, we have:
\begin{eqnarray}
\nonumber
 p^{2\alpha -1} \overline{\chi}(B) &=& \overline{\chi}(L_0) \chi(A),\\
 \label{gr_cond_eqn5}
 p^{2\alpha -1} \chi(B) &=& \chi(L_0) \overline{\chi}(A). 
\end{eqnarray}
Similarly, using Equation (\ref{gr_cond_eqn4}), we get:
\begin{eqnarray}
 \label{gr_cond_eqn6}
 p^{2\alpha} \chi(A) &=& \chi(L_0) \overline{\chi}(B). 
\end{eqnarray}

To show part 1, suffices to show that the left hand side of 
part 1 and right hand side
of part 1 have the same character values; 
the equation will follow by Fourier
inversion in $K$. Note that for any nonprincipal character, 
Equation (\ref{gr_cond_eqn6}) shows that the left hand side of
 part 1 and the right
hand side of part 1, do indeed, have the same character 
values. Suffices to show that
the same holds for the principal character.

The principal character of the left hand side of 
part 1 is $p^{2\alpha}\chi_0(A)$. 
On the other hand, the principal character of the right hand side is:
\begin{eqnarray}
\nonumber
 \chi_0( \chi_0(A) K(x) + L_0(x) *B(x^{-1})) &=& 
	\chi_0(A)p^{2\alpha} + \chi_0(L_0)\overline{\chi}(B) \\
\nonumber
 &=& \chi_0(A) p^{2\alpha},
\end{eqnarray}
where we used $\chi_0(L_0) = 0$. This shows that part 1, does 
indeed, agree on both
sides at the principal character. Thus, part 1 follows.

To show part 2, we use a similar analysis. By using 
Equation (\ref{gr_cond_eqn5}), we can deduce that the left
hand side of part 2 has the same character value as the
right hand side whenever the character used is nonprincipal.
Thus, it suffices to show the same for the principal character $\chi_0$.
Clearly, the principal character of the left hand side 
of part 2 is $p^{2\alpha} \chi_0(B)$. On the other
hand, the principal character of the right hand side
is:
\begin{eqnarray}
 \nonumber
  \chi_0( \chi_0(B) K(x) + pL_0(x) *A(x^{-1})) &=& 
	\chi_0(B)p^{2\alpha} + p\chi_0(L_0)\overline{\chi}(A) \\
\nonumber
 &=& \chi_0(B) p^{2\alpha}.
\end{eqnarray}
Thus, part 2 follows.

Clearly, part 3 is Equation (\ref{gr_cond_eqn1}) and part 4 is
Equation (\ref{gr_cond_eqn3}). 

We will show parts 5, 6, and 7 together. First, we will show
that $a_i$ is odd for $i=0,\ldots,r'$. This will show 
that: all the coefficients of $A$ are odd; $a_0$ is odd; 
$\chi_0(B)$ is odd, by part 4; $\chi_0(A)$ is odd, since
it is the sum of an odd number of odd numbers; and
$b_0$ is odd, by part 3. Finally, we will show
that $b_i$ is odd for $i=1,\ldots,r'$.

Consider $a_i = \LRGen{d_i',\underline{j}}$. Clearly, 
$d_i'$ is a vector of $\pm 1$s of length $r' = p^{2\alpha} + \cdots + p + 1$.
Let $x$ be the number of entries in $d_i'$ that are $+1$, and let $y$ be
the number of entries of $d_i'$ that are $-1$. Clearly,
\begin{eqnarray}
 \nonumber
  x-y &=& a_i,\\
 \nonumber
  x+y &=& r'.
\end{eqnarray}
Note that $r' = 1$ mod $2$. Hence, the above equations mod $2$ give
$a_i = x-y = x+y = r' = 1$. Thus, $a_i$ is odd.

Suffices to show $b_i$ is odd for $i=1,\ldots,r'$. By
manipulating Equation (\ref{gr_cond_eqn4}) algebraically we
 can deduce: 
\begin{eqnarray}
 \nonumber
 p^{2\alpha} a - \epsilon_0 p^{\alpha-1} b_0 \underline{j} &=&
	L b.
\end{eqnarray}
By multiplying the above equation by $L$ on both sides, and
simplyfing $L^2$ by using proposition \ref{SpcCaseP4}; we deduce.
\begin{eqnarray}
 \label{SpcCaseE3}
 L( p^{2\alpha} a - \epsilon_0 p^{\alpha-1} b_0 \underline{j} )
 &=& p^{2\alpha -1}(p^{2\alpha} I - J) b.
\end{eqnarray}
Note that mod $2$, 
\begin{eqnarray}
\nonumber
p^{2\alpha} a - \epsilon_0 p^{\alpha-1} b_0 
\underline{j} &=& \underline{j} - \underline{j} \\
 \nonumber
 &=& 0,
\end{eqnarray}
hence, Equation (\ref{SpcCaseE3}) mod $2$ implies:
\begin{eqnarray}
 \nonumber
 0 &=& (I - J)b \\
 \nonumber
  &=& b - \chi_0(B)\underline{j}\\
 \nonumber
  &=& b - \underline{j}. 
\end{eqnarray}
Thus, all the entries of $b$ are odd. 
\end{proof}
 
We leave as an open problem to determine the feasibility of the elements
$A(x),B(x) \in \Z[K]$ derived in proposition \ref{GSHDS_SpcCsP1}. 
We close this section with extra conditions when $\alpha = 1$.

\begin{prop}
Under the assumptions of Proposition (\ref{GSHDS_SpcCsP1}), assume that $\alpha =1$. Then,
\begin{enumerate}
 \item The set $D'' = D \cap L$ is a GSHDS.
 \item The following holds, $\nu_{G,H'}(D) = p^2 d_{D\cap L}$, where $H' = \{ g \in G \mid  p \cdot g = 0 \}$.
\end{enumerate}
\end{prop}

\begin{proof}
We proceed using a similar argument as that of Proposition (\ref{GSHDS_SpcCsP1}). By considering
Equation (\ref{gr_cond3}), we conclude that $p$ divides all entries of $A_{L,L_1}d_{D \cap L}$. Hence,
by Proposition (\ref{prop3}) part 3, $D \cap L$ is a GSHDS. This shows part 1.

Clearly, a similar argument shows that
$\overline{D} \cap \overline{L}$ is also a GSHDS. Now, consider the projection $\pi' = \mu_p : G \rightarrow p \cdot G = L$ 
and apply Proposition (\ref{prop8}) to the pair $D, \pi'$. Clearly, 
\begin{eqnarray}
\nonumber
 A_{L,L_1} \nu_{G,H'}(D) &=& df_G(D, L)
\end{eqnarray}
Where $\nu_{G,H'}(D)$ are the difference intersection numbers of $D$ with respect to $H'$, and $df_G(D,L)$ are the
difference coefficients of $D$ with respect to the characters of $L$ extended to $G$ by $\chi \rightarrow \chi \circ \pi'$.  
By Proposition (\ref{prop3}) part 3, $df_G(D,L) = p^{\alpha} d_{\overline{D} \cap \overline{L}}$ where $d_{\overline{D} \cap \overline{H}}$ is
the $\pm 1$s representation of the QRS $\overline{D} \cap \overline{L}$, and $\overline{D}$ is the dual GSHDS of $D$ introduced in
Corollary (\ref{corDual}). Thus, we have,
\begin{eqnarray}
\nonumber
 A_{L,L_1} \nu_{G,H'}(D) &=& p^3 d_{\overline{D} \cap \overline{L}}
\end{eqnarray}
Consider,
\begin{eqnarray}
 \nonumber
 p^2 \nu_{G,H'}(D) &=& A_{L,L_1}^2 \nu_{G,H'}(D) \\
 \nonumber
  &=& p^3 A_{L,L_1} d_{\overline{D} \cap \overline{L} } \\
 \nonumber
  &=& p^3 p d_{D \cap L} \text{  because } \overline{D}\cap \overline{L} \text{ is a GSHDS}\\
 \nonumber
  &=& p^4 d_{D \cap L},
\end{eqnarray}
Where, we have used the fact that $\overline{\overline{D}} = D$. Hence, the second part follows.
\end{proof}

\section{Special $p$-divisibility Conditions}
We close this article by showing special $p$-divisibility conditions for
the Difference Intersection Numbers $\nu_{G,L}(D)$ for special families of
groups $L$.

\begin{prop}
Let $D$ be a GSHDS in $G$. Assume that $exp(G) = p^s$, and $k \leq s-1$.
Consider the map $\mu_{p^k} : G \rightarrow G$ by $\mu_{p^k}(g) = p^k \cdot g$. Let $ L = Ker(\mu_{p^k}) = \{ g \in G \mid p^k \cdot g = 0 \}$, and
$H = \mu_{p^k}(G) = p^k \cdot G$. Then,
\begin{enumerate}
 \item There are integer constants $a_{p^k}, b_{p^k}, c_{p^k}$ such that,
\begin{eqnarray}
 \nonumber
  D(x)^{p^k} &=& c_{p^k}[1] + a_{p^k} D(x) + b_{p^k} D(x^{n_0})
\end{eqnarray}
 \item Let $D_H = D \cap H$, assume the coset decomposition $G = L \cup g_1 L \cup \cdots \cup g_r L$. Let $h_i = \mu_{p^k}(g_i)$, and assume that 
$h_{i_1}, \ldots, h_{i_t}$ is a set of $H_1 = (\Z/ exp(H)\Z)^*$ orbit representatives. Then, the following holds in $\Z[H]$,
\begin{eqnarray}
\nonumber
  (a_{p^k} - b_{p^k}) (D_H(x) - D_H(x^{n_0})) &=& \sum_{k=1}^t \nu_{G,L}(g_{i_k},D) ( O_{h_{i_k}}(x) - O_{h_{i_k}}(x^{n_0}) )\\
 \nonumber
	& &  + p^k ( C_H(x) - C_H(x^{n_0}))
\end{eqnarray}
 Where $O_{h_i}$ is the $H_2 = (\Z/ exp(H)\Z)^{*2}$ orbit of $h_i \in H$, and $C_H \in \Z[H]$.
 \item The exact power of $p$ dividing $a_{p^k} - b_{p^k}$ is $p^k$.
 \item The Difference Intersection numbers $\nu_{G,L}(h_i,D)$ are divisible by $p^k$.
\end{enumerate}
\end{prop}

\begin{proof}
We will consider the association scheme $A(G) = \LRGen{[1],D(x),D(x^{n_0})}$. By Proposition (\ref{prop2_2}) part 1, this  
association scheme is $3$-dimensional. Thus, for each $p^k$, there are integers $c_{p^k}, a_{p^k}, b_{p^k}$ such that,

\begin{eqnarray}
\nonumber
 D(x)^{p^k} &=& c_{p^k}[1] + a_{p^k} D(x) + b_{p^k} D(x^{n_0})\\
\nonumber
 (D(x^{n_0}))^{p^k} &=& c_{p^k}[1] + b_{p^k} D(x) + a_{p^k} D(x^{n_0})
\end{eqnarray}
Hence, part 1 follows. Clearly,
\begin{eqnarray}
\nonumber
 D(x)^{p^k} - (D(x^{n_0}))^{p^k} &=& (a_{p^k} - b_{p^k}) ( D(x) - D(x^{n_0}))
\end{eqnarray}
Also, by the multinomial theorem applied to $D(x)^{p^k}$ in $\Z[G]$,
\begin{eqnarray}
\nonumber 
D(x)^{p^k} &=& \mu_{p^k} \cdot D(x) + p^k C(x)\\
\nonumber
(D(x^{n_0}))^{p^k} &=& \mu_{p^k} \cdot D(x^{n_0}) + p^k C(x^{n_0})
\end{eqnarray}
Where $\mu_{p^k} \cdot D(x) = \sum\{ x^{p^k \cdot g} \mid g \in D \}$, and $C(x) \in \Z[G]$. 
Thus, 
\begin{eqnarray}
\nonumber
\mu_{p^k} \cdot ( D(x) - D(x^{n_0})) + p^k ( C(x) - C(x^{n_0}) )   &=&
 	D(x)^{p^k} - (D(x^{n_0}))^{p^k} \\
\nonumber
 &=& ( a_{p^k} - b_{p^k}) ( D(x) - D(x^{n_0}) )	 
\end{eqnarray}
We restrict the previous equation to $H$. Thus, we deduce a similar equation
in $\Z[H]$,
\begin{eqnarray}
\label{gshds_third_eq1}
 ( a_{p^k} - b_{p^k} ) ( D_H(x) - D_H(x^{n_0}) ) &=& \mu_{p^k} \cdot ( D(x) - D(x^{n_0}) ) \\
 \nonumber
 & & + p^k (C_H(x) - C_H(x^{n_0}) )
\end{eqnarray}
Clearly,
\begin{eqnarray}
\nonumber
 \mu_{p^k} \cdot ( D(x) - D(x^{n_0}) ) &=& 
\sum_{i=1}^r \nu_{G,L}(g_i,D) x^{[h_i]}
\end{eqnarray}
Note that for $n \in (\Z/ exp(H)\Z)^*$,
\begin{eqnarray}
\nonumber
 \nu_{G,L}(n \cdot g,D) &=& \QRSym{n} \nu_{G,L}(g,D)
\end{eqnarray}
Hence, part 2 follows from Equation (\ref{gshds_third_eq1}). 

Now, we proceed to show part 3. Let $\chi \in \overline{G}$ such that,
\begin{eqnarray}
\nonumber
\chi(D) &=& \frac{-1 + p^{\alpha} \sqrt{ \QRSym{-1}p}}{2}\\
\nonumber
\chi(D^{(n_0)}) &=& \frac{-1 - p^{\alpha} \sqrt{ \QRSym{-1}p}}{2}
\end{eqnarray}
Clearly, by using the binomial theorem,
\begin{eqnarray}
\nonumber
 \chi(D)^{p^k} - \chi(D^{(n_0)})^{p^k} &=&
\frac{1}{2^{p^k-1}} \sum_{f=0}^{\frac{p^k -1}{2}} 
			\binom{p^k}{2f+1} p^{\alpha (2f+1) + f}\QRSym{-1}^f 
				\sqrt{\QRSym{-1}p}
\end{eqnarray}
Note that,
\begin{eqnarray}
\nonumber
 \chi(D)^{p^k} - \chi(D^{(n_0)})^{p^k} &=&
	(a_{p^k} - b_{p^k}) ( \chi(D) - \chi(D^{(n_0)}) ) \\
\nonumber
	&=& (a_{p^k} - b_{p^k}) p^{\alpha} \sqrt{\QRSym{-1}p}
\end{eqnarray}
Hence, by equating the last two equations, we get,
\begin{eqnarray}
\nonumber
a_{p^k} - b_{p^k} &=& 
\frac{1}{2^{p^k-1}} \sum_{f=0}^{\frac{p^k -1}{2}} 
			\binom{p^k}{2f+1} p^{(2\alpha +1)f}\QRSym{-1}^f 
\end{eqnarray}
Clearly, every term in the sum above is divisible by $p^k$ and the first term
of the sum is $p^k$. Hence, the highest power of $p$ dividing $a_{p^k} - b_{p^k}$ is $p^k$. This shows part 3.

Part 4 follows from parts 2 and 3.
\end{proof}

\section{Conclusion}

We close this article by thanking the referees for their valuable criticism and for pointing out the connection between GSHDSs and Lam's generalized cyclic difference sets.

\bibliographystyle{plain}
\bibliography{thesis_third}

\end{document}